\documentclass[a4paper]{article}

\usepackage[utf8]{inputenc}
\usepackage[T1]{fontenc}
\usepackage{bbm}
\usepackage{xspace}
\usepackage[svgnames]{xcolor}
\usepackage{mathtools,dsfont}
\usepackage{amsmath,amssymb}
\usepackage{amssymb,amsmath,amsthm}

\usepackage{authblk}
\usepackage[ruled,vlined,commentsnumbered]{algorithm2e}

\usepackage[natbib=true,style=numeric,sortcites=true]{biblatex}
\usepackage{booktabs}
\usepackage{siunitx}

\usepackage{tikz,tikzscale}

\usepackage{caption,subcaption}
\usepackage{pgfplots}
\pgfplotsset{compat=1.15}
\usepgfplotslibrary{groupplots}
\usepgfplotslibrary{statistics}

\usepackage{tabularx}
\usepackage{todonotes}

\usepackage[hidelinks]{hyperref}
\usepackage[capitalize]{cleveref}

\newcommand{\R}{\mathds{R}}

\newcommand{\N}{\mathds{N}}

\newcommand{\st}{\textrm{s.t.\xspace}}
\newcommand{\ie}{i.e.,\xspace}

\DeclareMathOperator{\TV}{TV}
\DeclareMathOperator{\LP}{LP}
\DeclareMathOperator{\crit}{crit}
\DeclareMathOperator{\low}{low}
\DeclareMathOperator{\diag}{diag}
\DeclareMathOperator{\argmin}{arg\,min}

\renewcommand{\epsilon}{\varepsilon}

\theoremstyle{plain}
\newtheorem{thm}{Theorem}[section]
\newtheorem{lem}[thm]{Lemma}
\newtheorem{prop}[thm]{Proposition}

\theoremstyle{definition}

\newtheorem{rem}[thm]{Remark}

\newtheorem{ass}{Assumption}

\crefname{ass}{Assumption}{Assumptions}
\crefname{lem}{Lemma}{Lemmas}
\crefname{cor}{Corollary}{Corollaries}

\crefrangeformat{line}{Lines #3#1#4~--~#5#2#6}

\newcommand{\define}{\coloneqq}
\newcommand{\enifed}{\eqqcolon}

\newcommand{\trustRadius}{\Delta}
\newcommand{\trustRadiusLP}{\Delta^{\LP}}

\newcommand{\trans}[1]{{#1}^{\text{T}}}

\newcommand{\trustRadiusLPMax}{\trustRadiusLP_{\max}}

\newcommand{\ratioParam}{\vartheta}

\newcommand{\ratioParamOpt}{\ratioParam^{\ast}_{\epsilon}}

\newcommand{\noisy}[1]{\tilde{#1}}

\newcommand{\normLP}[1]{\ensuremath{\|#1\|_{\LP}}}

\newcommand{\normFrob}[1]{\ensuremath{\|#1\|_{F}}}

\newcommand{\onehalf}{\tfrac{1}{2}}
\newcommand{\Lip}{L}
\newcommand{\LipOmega}{\Lip^{\omega}}
\newcommand{\LipF}{\Lip^{F}}
\newcommand{\LipFF}{\Lip^{F'}}
\newcommand{\LipEll}{\Lip^{\ell}_{\epsilon}}

\newcommand{\FactConst}{M_{0}^{\epsilon}}
\newcommand{\FactLin}{M_{1}^{\epsilon}}
\newcommand{\FactQuad}{M_{2}^{\epsilon}}

\newcommand{\imageNoise}{\varepsilon_{\text{img}}}

\newcommand{\email}[1]{\href{mailto:#1}{\nolinkurl{#1}}}

\newcommand{\thetitle}{Convergence of Successive Linear Programming Algorithms for Noisy Functions}

\newcommand{\proofAppendix}{\begin{proof}The proof is in the appendix.\end{proof}}

\newcommand{\theauthone}{Christoph Hansknecht}
\newcommand{\theauthtwo}{Christian Kirches}
\newcommand{\theauththree}{Paul Manns}

\tikzset{
  CoordNode/.style={
    draw,fill,circle,inner sep=0.5
  }
}

\colorlet{ClassicalColor}{blue}
\colorlet{StabilizedColor}{red}

\colorlet{NoisyColor}{blue}
\colorlet{NoiselessColor}{red}

\pgfplotsset{
  Classical/.style={
    mark=none,
  },
  Stabilized/.style={
    mark=none,dashed
  },
  Iterations/.style={
    xmin=0,
    xmax=50,
    xtick={0, 10, ..., 50}
  }
}

\hypersetup{
  pdftitle={\thetitle},
  pdfauthor={\theauthone, \theauthtwo, \theauththree}
}

\newcommand{\largeNum}[1]{$\num[round-mode=places,round-precision=0,scientific-notation=true]{#1}$}

\begin{document}

\title{\thetitle}

\author[1]{\theauthone}
\affil[1]{\small Institute for Mathematical Optimization, TU~Braunschweig,\protect\\
  \texttt{\{\href{mailto:c.hansknecht@tu-braunschweig.de}{c.hansknecht},\href{mailto:c.kirches@tu-braunschweig.de}{c.kirches}\}@tu-braunschweig.de}}

\author[1]{\theauthtwo}

\author[2]{\theauththree}
\affil[2]{\small Chair of Numerical Analysis and Optimization, TU~Dortmund~University\protect\\
  \email{paul.manns@tu-dortmund.de}}

\date{\today}

\maketitle

\begin{abstract}
  Gradient-based methods have been highly successful for solving a variety
of both unconstrained and constrained nonlinear optimization problems.
In real-world applications, such as optimal control or machine 
learning, the necessary function and derivative information may
be corrupted by noise, however. Sun and Nocedal have recently
proposed a remedy for smooth unconstrained problems by means
of a stabilization of the acceptance criterion for computed iterates,
which leads to convergence of the iterates of a
trust-region method to a region of criticality \cite{noisy_trust_region}.

We extend their analysis to the successive linear programming
algorithm \cite{sleqp,sleqp_convergence} for unconstrained optimization problems
with objectives that can be characterized as the composition of a polyhedral function
with a smooth function, where the latter and its gradient
may be corrupted by noise. This gives the flexibility to cover,
for example, (sub)problems arising image reconstruction
or constrained optimization algorithms.

We provide computational examples that illustrate the findings
and point to possible strategies for practical determination
of the stabilization parameter that balances the size of
the critical region with a relaxation
of the acceptance criterion (or
descent property) of the algorithm.
\end{abstract}

\section{Introduction}

Handling non-smoothness is an ubiquitous research
question in nonlinear optimization because it
arises naturally in different areas, for example,
penalty functions for constrained optimization
\cite{nlp_book}, statistical data analysis
and signal processing 
\cite{tibshirani1996regression,candes2006stable},
and neural network architectures
\cite{fukushima1975cognitron}.
In this work we study the convergence properties of successive linear programming algorithms
to solve the optimization problem
\begin{equation}
  \label{eq:convex_problem}
  \tag{P}
  \min_{x \in \R^n} \: \phi(x) \define \omega(F(x)),
\end{equation}
where $\omega : \R^{p} \to \R$ is convex and Lipschitz 
continuous with polyhedral epigraph, and
$F : \R^{n} \to \R^{p}$ is twice continuously differentiable. Moreover, we assume that $F$
and its gradient can only be accessed inexactly
so that their evaluations are corrupted by noise.
This and similar problems have been studied in the literature, see, for example, \cite{aspelmeier2016local,cartis2011evaluation,apkarian2016nonsmooth,boyd2011distributed,nesterov2013gradient}
and the references therein.

Many optimization problems can be 
formulated in terms of problem~\eqref{eq:convex_problem}
such as the Lagrangian form 
\[ \min_{x \in \R^n} \|y - Ax\|^2_2 + \beta \|x\|_1 \]
of the 
famous LASSO problem \cite{tibshirani1996regression,santosa1986linear}
with $A \in \R^{m\times n}$, $\beta > 0$,
and $y \in \R^m$ that is particularly popular among data scientists for
sparse parameter identification in over-parameterized
models. It is an instance of a 
subproblem for the exact penalty method
for general nonlinear constrained optimization problems that read
\begin{equation}
  \label{eq:nonlinear_problem}
  \tag{NLP}
  \min_{x \in \R^n} \: f(x) \:\: \st \:\: g(x) \leq 0, \: h(x) = 0
\end{equation}
with an objective $f : \R^n \to \R$ and constraints $g : \R^n \to
\R^{m}$, $h : \R^n \to \R^{k}$. 

The problem \eqref{eq:nonlinear_problem}
may be solved by minimizing a non-smooth
exact penalty function of the form
\begin{equation}
  \label{eq:penalty_func}
  \phi(x, \nu) \define f(x) + \nu \left\|\trans{\left( \trans{g(x)}_{+}, \trans{h(x)} \right)}\right\|_{1},
\end{equation}
where $y_{+} \define \max(y, 0)$.  In fact, strict local solutions
of~\eqref{eq:nonlinear_problem} are local minimizers of $\phi(x, \nu)$
for a sufficiently large value of $\nu$ if $g$ and $h$
are smooth and satisfy the
Mangasarian--Fromovitz constraint
qualification \cite[Theorem~4.4]{exact_penalty}, \cite[Theorem~17.3]{nlp_book} at the
respective points. In this case, the penalty function $\phi$ can be
expressed as $\omega(F(x))$, where
$F(x) \define \trans{(f(x), \trans{g(x)}, \trans{h(x)})}$ is
smooth and
$\omega(x, y, z) \define x + \nu \|\trans{(\trans{y_{+}}, \trans{z})}\|_{1}$
is convex and polyhedral. Besides problems of type~\eqref{eq:nonlinear_problem}, a
variety of other problems, such as linear or nonlinear fitting
problems or unconstrained smooth optimization problems can be
formulated in terms of~\eqref{eq:convex_problem} as well.

\paragraph{Noisy functions}
The combination of unconstrained optimization with
noisy observations has recently been examined~\cite{a_noise_tolerant,analysis_of_the_bfgs_method_with_errors,noisy_trust_region}.
The authors consider the minimization of a smooth function $\phi : \R^{n} \to \R$ while only having access to
\[
    f(x) = \phi(x) + \epsilon(x) \text{ and }
    g(x) = \nabla{\phi(x)} + e(x),
\]
where the only assumptions on $\varepsilon$
and $\epsilon$ are that both $|\epsilon|$
and $\|e\|$ are uniformly bounded. Consequently, it is not 
generally possible to generate a sequence $\{x_k\}$ of iterates converging to a local optimum or stationary point
of $\phi$. Intuitively, while the gradient noise $e$ is small compared
to $\nabla{\phi}$, the direction $g$ is a suitable search direction
with respect to $\phi$. This allows for the use of an Armijo-like globalization
strategy or, in case of \cite{noisy_trust_region}, a trust-region method, where the noise is handled by
stabilizing the reduction ratio, which
is of course closely related to the Armijo
condition.
As soon as a region is reached where the
noise produced by $\epsilon$ and $e$ becomes too large relative to
$\phi$ and $\nabla \phi$ respectively, no further progress can be expected
and the algorithm may stall. However, this critical region is visited
infinitely often and once reaching it,
the algorithm does not produce objective 
values much larger than the objective 
values attained in the critical region.
The authors also study the problem of 
adapting quasi-Newton methods to the noisy setting.

\paragraph{Contribution} We build on the ideas in 
\cite{a_noise_tolerant,analysis_of_the_bfgs_method_with_errors,noisy_trust_region}
and consider the non-smooth problem~\eqref{eq:convex_problem} in a
setting, where function and derivative evaluations are only available
as \emph{noisy} observations. As the authors in 
\cite{a_noise_tolerant,analysis_of_the_bfgs_method_with_errors,noisy_trust_region},
we assume the following noise model:
Rather than being able to evaluate $F$ and its derivative $F'$
directly, we only have access to
\begin{equation*}
    \noisy{F}(x) \define F(x) + \delta_F(x) \text{ and }
    \noisy{F'}(x) \define F'(x) + \delta_{F'}(x).
\end{equation*}
These proxies consist of the original functions $F$ and $F'$ as well
as error functions $\delta_{F} : \R^{n} \to \R^{p}$ and
$\delta_{F} : \R^{n} \to \R^{p \times n}$. We assume that the noise is
uniformly bounded via $\|\delta_F(x)\| \leq \epsilon_F$ and
$\|\delta_{F'}(x)\| \leq \epsilon_{F'}$ for all $x \in \R^n$, where $\|\cdot\|$
is the 2-norm in $\R^{n}$ induced by the standard scalar
product $\langle \cdot, \cdot \rangle$. We refer to
$\epsilon_{F}$ and $\epsilon_{F'}$ as the \emph{noise levels}
of $\noisy{F}$ and $\noisy{F'}$ respectively.
In terms of the problem~\eqref{eq:nonlinear_problem}, this is
tantamount to noise in the objective $f$, the constraints $g, h$, and
their respective derivatives.  Contrary to this, we assume that the
function $\omega$ does not suffer from any noise. What is more, we
presume that the structure of $\omega$ is well understood in the sense
that, for example, its Lipschitz constant is known, which is certainly
the case for the penalty function in~\eqref{eq:penalty_func}.

In order to solve optimization problems
of the form \eqref{eq:convex_problem}, we propose a
trust-region algorithm leaning on the successive
linear programming template proposed in
\cite{sleqp} and a convergence analysis that builds on the
ideas in
\cite{sleqp_convergence,a_noise_tolerant,analysis_of_the_bfgs_method_with_errors}.
Specifically, we use a stabilization of the iterate acceptance test in order
to assert that a neighborhood of a stationary point is visited infinitely often
by the iterates produced by the algorithm.
The polyhedral structure of $\omega$ is handled by first
solving a linear program in order to determine a direction
for a subsequent Cauchy point determination. This can
also be interpreted as an active set determination
for the corresponding kinks of the
polyhedral epigraph of $\omega$.

We also provide computational examples that illustrate
the theoretical results and the practical behavior
of the algorithm. Moreover, the results point to
open questions and possible approaches regarding
the choice of the correct stabilization parameter in
the acceptance test.

\paragraph{Structure of the Remainder}
We introduce the successive linear programming algorithm and 
the modified acceptance test in
\cref{sec:algorithm}. The asymptotics of the algorithm are analyzed in
\cref{sec:convergence}. We provide computational examples and the corresponding
results in \cref{sec:numerics}. We draw a conclusion in \cref{sec:conclusion}.

\section{A Noise-tolerant Successive Linear Programming Algorithm}\label{sec:algorithm}

In the noisy setting, we cannot expect to find the true 
optimum or stationary points of $\phi$, since we do not have 
access to $F$ and $\nabla F$. Specifically, in a small 
region around the true optimum $x^{*}$, $\noisy{F}$
and $\noisy{F'}$ may oscillate by amounts of
$\pm \varepsilon_{F}$ and $\pm \varepsilon_{F'}$,
thereby making its evaluations unreliable. This impairs 
globalization strategies in nonlinear programming because
their acceptance tests require reliable evaluations of
$F$ and a model function involving $\nabla F$.

In the non-noisy regime, a trust-region method produces
a sequence $\{x_k\}$ of iterates by assembling and
subsequently optimizing model functions $q_k : \R^n \to \R$, 
yielding a step $d$. The quality of $d$ is
determined according to the reduction ratio
\begin{equation*}
  \label{eq:normal_reduction_ratio}
  \rho_k \define \frac{\phi(x_k) - \phi(x_k + d_k)}{\phi(x_k) - q_k(d_k)},
\end{equation*}
which is used to determine whether or not the step will be
accepted.  However, in the noisy setting, we only have access to
$\noisy{F}$ leading to a composite function $\noisy{\phi}$ defined as
$\noisy{\phi}(x) \define \omega(\noisy{F}(x))$.
While we can build a model
$\noisy{q}_k : \R^{n} \to \R^{p}$ which coincides with $\noisy{F}$ at
$x_k$, we cannot control the numerator
$\noisy{\phi}(x_k) - \noisy{\phi}(x_k + d_k)$. Indeed, if we reduce
the trust region, sending $d_k$ to zero, the denominator of $\rho_k$
will tend to zero while the numerator will oscillate by up to
$\pm 2 \epsilon_F$, making the ratio unreliable. To alleviate this
problem, we turn towards a recent adaptation \cite{noisy_trust_region}
of trust-region methods to solving the noisy counterpart of
smooth, unconstrained problems like~\eqref{eq:convex_problem}.
The authors of \cite{noisy_trust_region} add a correction term,
that is a positive constant $\ratioParam > 0$, to both the numerator
and denominator of the reduction ratio $\rho_k$ to mitigate the effect
of noisy evaluations, yielding a modified ratio
\begin{equation*}
  \label{eq:noisy_reduction_ratio}
  \noisy{\rho}_k \define
  \frac{\noisy{\phi}(x_k) - \noisy{\phi}(x_k + d_k) + \ratioParam }
       {\noisy{\phi}(x_k) - \noisy{q}_k(d_k) + \ratioParam}.
\end{equation*}
The parameter $\ratioParam$ can then be chosen according to the noise
levels $\epsilon_{F}$ and $\epsilon_{F'}$ in order to stabilize the
ratio. As we will see, this means that for $\ratioParam$ large enough,
the iterates
of the successive linear programming algorithm converge to
a critical region around stationary point. The downside is that this region
grows with $\ratioParam$ and the algorithm also accepts steps that do not
improve the objective.

Apart from this adjustment, we follow the algorithmic approach
in~\cite{sleqp_convergence}. Specifically, we use the following
\emph{partially linearized} and \emph{quadratic} models at $x_k \in \R^n$
\begin{equation*}
  \begin{aligned}
    \noisy{\ell}_k(d) &\define \omega \left( \noisy{F}(x_k) + \noisy{F}'(x_k) d  \right) \mathrm{ and}\\
    \noisy{q}_k(d) &\define \noisy{\ell}_k(d) + \onehalf \langle d, B_k d \rangle,
  \end{aligned}
\end{equation*}
where the $B_k \in \R^{n \times n}$ are symmetric (not necessarily
positive definite) approximations of the curvature of $\omega \circ F$.

\paragraph{Algorithm}
Based on the models above, our noise-tolerant approach to
solving~\eqref{eq:convex_problem} is laid out in
\cref{alg:noisy_sleqp}. In each iteration, an initial step
$d^{\LP}$ is computed in~\cref{line:lp_step} by solving the problem
\begin{equation*}
  \min_{\normLP{d} \leq \trustRadiusLP_k} \noisy{\ell}_k(d),
\end{equation*}
where $\normLP{\cdot}$ is some norm on $\R^{n}$. We highlight that,
from an algorithmic point of view, it is advantageous to cast this
problem as a linear program, which can be solved using
state-of-the-art LP solvers~\cite{gurobi,HiGHS}. To this end, both the epigraph of $\omega$
and the feasible region given in terms of $\normLP{\cdot}$ should be
polyhedral as is the case, for example, for the $\ell^1$- and $\ell^\infty$-norms.

Due to the equivalence of norms in $\R^n$ there exists a
constant $\gamma > 0$ such that for each $d \in \R^{n}$ it holds that
\begin{equation}
  \label{eq:lp_norm_equiv}
  \|d\| \leq \gamma \normLP{d}.
\end{equation}
The algorithm proceeds to compute a Cauchy step $d^{C}_k$ in
\crefrange{line:cauchy_step_start}{line:cauchy_step_end}. To this end, it employs a
line search initialized with a step size sufficiently small to
ensure that the Cauchy step falls into the trust region bounded by
$\trustRadius_k$. During the line search the step size is shortened
by a factor of $0 < \tau < 1$ until the quadratic reduction achieved 
by the Cauchy point is within a factor of $0 < \eta < 1$ of its
linear reduction.

The actual step $d_k$ to be taken in~\cref{line:actual_step} can be
different from the Cauchy step $d^{C}_k$, provided that it improves
upon the quadratic reduction of $d^{C}_k$. This gives some algorithmic
flexibility, allowing for the computation of Newton-type steps in
order to achieve local quadratic convergence. Based on the stabilized
reduction ratio $\noisy{\rho}_k$ computed
in~\cref{line:reduction_ratio}, the step is either accepted
(\crefrange{line:accept_begin}{line:accept_end}) or rejected
(\crefrange{line:reject_begin}{line:reject_end}) according to an
acceptance threshold of $\rho_u > 0$. Additionally, the
trust-region radii $\trustRadius$ and $\trustRadiusLP$ are
adjusted based on $\noisy{\rho}_k$:
\begin{enumerate}
\item
  The value of $\trustRadius$ is increased or decreased based
  on whether $\noisy{\rho}_k$ achieves a value of at least $\rho_{s}$.
  The decrease is such that the new trust-region radius is at
  most $\kappa_u < 1$ times as large as the previous one, thereby
  ensuring a true reduction, while being at least $\kappa_l \|d_k\|$
  (with $0 < \kappa_l \leq \kappa_u$) in order to prevent an immediate collapse
  of the trust region.
\item
  If $\noisy{\rho}_k$ achieves at least $\rho_u$, the LP trust-region radius
  $\trustRadiusLP$ is increased beyond $\normLP{d^{C}_k}$, as long
  as it does not exceed the upper bound of ${\trustRadiusLPMax \geq 1}$.
  The new LP trust-region radius is also only increased beyond
  $\trustRadiusLP$ if the full LP step $d^{\LP}$ was
  accepted (\ie $\alpha_k = 1$), indicating that the partially
  linearized model $\noisy{\ell}_k$ is a good approximation
  of $\noisy{\Phi}_k$ across the entire LP trust region.
  If $\noisy{\rho}_k$ falls short of $\rho_u$, $\trustRadiusLP$ is
  decreased while being kept within a factor of $\theta > 0$
  of $\normLP{d_k}$.
\end{enumerate}
\begin{rem}
When applied to problem~\eqref{eq:nonlinear_problem}, 
\cref{alg:noisy_sleqp} uses the strategies introduced in~\cite{sleqp}, 
which form the basis of the active set
method in the highly successful \textsc{Knitro} 
code~\cite{knitro}, which combines sequential linear programming
with equality constrained quadratic programming approaches in order
to achieve robust performance over a range of large-scale nonlinear 
programming problems.
\end{rem}
\begin{algorithm}
  \DontPrintSemicolon

\newcommand{\Assign}[2]{$#1 \leftarrow$ #2}

\LinesNumbered
\SetCommentSty{itshape}
\DontPrintSemicolon
\SetKwComment{Comment}{$\triangleright$\ }{}
\SetKwInOut{Input}{Input}
\SetKwInOut{Output}{Output}
\SetKwInOut{Parameters}{Parameters}
\SetKwFor{Until}{until}{}

\Input{Functions $\omega$, $\noisy{F}$, $\noisy{F'}$, \\
  Initial point $x_0 \in \R^n$, \\
  Initial trust region radii $0 <\trustRadiusLP_0 \leq \trustRadiusLPMax$, $0 <\trustRadius_0$}
\Parameters{Acceptance thresholds $0 < \rho_u < \rho_s < 1$, \\
  Step adjustments $0 < \kappa_l \leq \kappa_u < 1$, $\theta > 0$, \\
  Cauchy line search parameters $0 < \eta < 1$, $0 < \tau < 1$, \\
  Ratio stabilizer $\ratioParam > 0$, \\
  Maximum LP trust region radius $\trustRadiusLPMax \geq 1$}
\Output{Primal point $x^{*} \in \R^{n}$}
\BlankLine
\Assign{k}{0} \;
\Until{Some termination criterion is satisfied}{

  \label{line:lp_step}
  Compute LP step
  \begin{equation*}
    d^{\LP}_k \gets \argmin_{\normLP{d} \leq \trustRadiusLP_k} \noisy{\ell}_k(d)
  \end{equation*}

  \label{line:cauchy_step_start}
  \Assign{\alpha_k}{$\min\left( 1, \trustRadius_k / \|d^{\LP}_k\| \right)$} \;

  \While{$\noisy{\phi}(x_k) - \noisy{q}_k(\alpha_k d^{\LP}_k) < \eta \left[ \noisy{\phi}(x_k) - \noisy{\ell}_k(\alpha_k d^{\LP}_k) \right]$}{
    \Assign{\alpha_k}{$\tau \alpha_k$} \;
  }

  \label{line:cauchy_step_end}
  \Assign{d^{C}_k}{$\alpha_k d^{\LP}_k$} \;

  \label{line:actual_step}
  \Assign{d_k}{Step $d$ such that $\|d\| \leq \Delta_k$ and $\noisy{q}_k(d) \leq \noisy{q}_k(d^{C}_k) $} \;

  \label{line:reduction_ratio}
  Compute stabilized reduction ratio
  \begin{equation*}
    \noisy{\rho}_k \gets \frac{\noisy{\phi}(x_k) - \noisy{\phi}(x_k + d_k) + \ratioParam}{\noisy{\phi}(x_k) - \noisy{q}_k(d_k) + \ratioParam}
  \end{equation*}

  \eIf(\Comment*[f]{Accept step}){$\noisy{\rho}_k \geq \rho_u$}{
    \label{line:accept_begin}
    Set $x_{k +1} \gets x_{k} + d_k$ \;
    \label{line:accept_end}
    Pick $\trustRadiusLP_{k + 1} \in \left[\normLP{d^{C}_k}, \trustRadiusLPMax \right]$ such that $\trustRadiusLP_{k + 1} \leq \trustRadiusLP_{k}$ if $\alpha_k < 1$ \;
  }(\Comment*[f]{Reject step}){
    \label{line:reject_begin}
    Set $x_{k +1} \gets x_{k}$ \;
    \label{line:reject_end}
    Pick
    $\trustRadiusLP_{k + 1} \in \left[\min(\theta \normLP{d_k}, \trustRadiusLP_k), \trustRadiusLP_k \right]$ \;
  }
  \eIf{$\noisy{\rho}_k \geq \rho_s$}{
    Set $\trustRadius_{k + 1} \geq \trustRadius_k$ \;
  }{
    Choose $\trustRadius_{k + 1} \in [\kappa_l \|d_k\|, \kappa_u \trustRadius_k]$ \;
  }  
  \Assign{k}{$k + 1$} \;
}

\Return{$x^{*} = x_k$} \;

  \caption{A noise-tolerant algorithm to minimize $\phi(x) = \omega(F(x))$}
  \label{alg:noisy_sleqp}
\end{algorithm}

\section{Convergence Analysis of \cref{alg:noisy_sleqp}}
\label{sec:convergence}

We begin our convergence analysis with the introduction of the 
standing assumptions and a recap of the relevant stationarity
concept for \eqref{eq:convex_problem}
in \cref{sec:assumptions}.
We analyze the criticality measure for this notion of
stationarity in the noisy setting
in \cref{sec:criticality_measure}.
We use these results to prove lower bounds on the trust-region
radii that occur in \cref{alg:noisy_sleqp}
in \cref{sec:lb_trust_region_radius}, which are then used
to obtain sufficient decrease and, as a consequence, convergence
of the produced iterates to critical regions
in \cref{sec:global_convergence_theorem}.

\subsection{Standing Assumptions and Stationarity}
\label{sec:assumptions}
In order to study the convergence properties of
\cref{alg:noisy_sleqp}, we make the following assumptions
regarding the functions $\omega$, $F$, and the matrices $B_k$ used in
the quadratic models $\noisy{q}_k$.
\begin{ass}
  \label{ass:omega_lipschitz}
  $\omega$ is Lipschitz-continuous with constant $\LipOmega$, \ie it holds
  for all $x,y \in \R^{n}$ that
  \begin{equation*}
    |\omega(x) - \omega(y)| \leq \LipOmega \|x - y\|.
  \end{equation*}
\end{ass}
\begin{ass}
  \label{ass:F_lipschitz}
  $F$ and $F'$ are Lipschitz-continuous with constants $\LipF$ and
  $\LipFF$, \ie it holds for all $x,y \in \R^{n}$ that
  \begin{equation*}
    \begin{aligned}
      \|F(x) - F(y)\| &\leq \LipF \| x - y\| \text{ and} \\
      \|F'(x) - F'(y)\| &\leq \LipFF \| x - y\|. \\
    \end{aligned}
  \end{equation*}
\end{ass}
\begin{ass}
  \label{ass:hess_bounded}
  The Hessian approximations $B_k$ are bounded, \ie there exists $\beta > 0$ such that
  for all $d \in \R^n$, $k > 0$ it holds that
  \begin{equation*}
    \langle d, B_k d \rangle \leq \beta \|d\|^2.
  \end{equation*}
\end{ass}
Our aim in the following is to find a local optimum of $\phi$. A
first-order necessary condition
(see~\cite[pp. 184]{practical_methods}) of optimality
for~\eqref{eq:convex_problem} states that $x^{*} \in \R^n$ can only be
a local optimum if
\begin{equation*}
  \max_{\lambda \in \partial \omega(F(x^{*}))}
  \langle \lambda, F'(x^{*}) d \rangle
  \geq 0 \quad \text{for all } d \in \R^n,
\end{equation*}
where $\partial \omega(z)$ denotes the subdifferential of $\omega$ at
$z \in \R^{p}$.
To measure the distance of the iterates $x_k$ to a critical point of
$\phi$, the authors of~\cite{sleqp_convergence} use the objective
reduction with respect to the partially linearized model, given by
\begin{equation*}
  \Psi_k(\Delta) \define \phi(x_k) - \min_{\normLP{d} \leq \Delta } \ell_k(d),
\end{equation*}
Clearly, since $d = 0$ is feasible, the reduction $\Psi_k(\Delta)$ is
always non-negative. On the other hand, the following results
establishes that a vanishing reduction over a trust region
of a normalized size is tantamount to reaching a critical point.
\begin{lem}[\cite{conditions_for_convergence}, Lemma 2.1]
  \label{lem:general_converge}
  Let $\{x_k\}_{k \in \N}$ converge to $x_{*}$. If
  \cref{ass:omega_lipschitz,ass:F_lipschitz}
  hold and
  \begin{equation*}
    \lim_{k \to \infty} \Psi_k(1) = 0,
  \end{equation*}
  then $x_{*}$ is a critical point of $\phi$.
\end{lem}

\subsection{Analysis of Model Function and Criticality Measure
in the Presence of Noise}\label{sec:criticality_measure}
Since we do not have access to the values of $F$ and $F'$ required to
compute $\Psi_k$, we define a noisy measure of criticality via
\begin{equation*}
  \noisy{\Psi}_k(\Delta) \define \noisy{\phi}(x_k) - \min_{\normLP{d} \leq \Delta } \noisy{\ell}_k(d).
\end{equation*}
This function is also non-negative and we analyze its properties and relationship to $\Psi_k$
below. Since $\delta_{F}$ cannot be assumed to be continuous, so can't $\noisy{\phi}$. 
This differs from the analysis in~\cite{sleqp_convergence}, where the Lipschitz-continuity of
$\phi$ is used to argue that the reduction ratio approaches one if the
trust-region radius is driven to zero.  We can, however, state that
the criticality measures $\Psi_k$ and $\noisy{\Psi}_k$
are related by the following approximation result:
when considering a fixed $x_k$, we claim that
$\noisy{\Psi}_k(1) \to \Psi_k(1)$ for $\varepsilon_F \to 0$ and
$\varepsilon_{F'} \to 0$ and that we also have convergence of the minimizers of the
convex programs in the definitions of $\noisy{\Psi}_k(1)$ and
$\Psi_k(1)$. This follows from the epi-convergence
of the functionals
\begin{equation*}
  \mathcal{L}_{\varepsilon_{F},\varepsilon_{F'}}(d) \define \noisy{\ell}_k(d) + i_{\normLP{d} \leq \Delta}(d)
  \quad \text{and} \quad
  \mathcal{L}_{0,0}(d) \define \ell_k(d) + i_{\normLP{d} \leq \Delta}(d),
\end{equation*}
where $i_A : \R^d \to \{0,\infty\}$ is the indicator function of $A \subset \R^d$, that is
$i_A(x) = \infty$ if $x \notin A$ and $i_A(x) = 0$ else.
We recall that the functionals $\mathcal{L}_{\varepsilon_{F},\varepsilon_{F}}$ epi-converge to $\mathcal{L}_{0,0}$
if and only if for all $d \in \R^n$ the inequalities
\begin{gather*}
\begin{aligned}
\mathcal{L}(d) &\le \liminf_{\varepsilon_{F},\varepsilon_{F'} \to 0} \mathcal{L}_{\varepsilon_{F},\varepsilon_{F'}}(d_\varepsilon)
\text{ for all sequences } d_\varepsilon \to d
\text{ and}\\
\mathcal{L}(d) &\ge \limsup_{\varepsilon_{F},\varepsilon_{F'} \to 0} \mathcal{L}_{\varepsilon_{F},\varepsilon_{F'}}(d_\varepsilon)
\text{ for some sequence } d_\varepsilon \to d
\end{aligned}
\end{gather*}
hold, see, for example,~\cite[\S~7]{rockafellar2009variational}, which is shown below.
\begin{prop}
\label{prop:noiseless_convergence}
Let $\Delta > 0$. Then the functionals $\mathcal{L}_{\varepsilon_{F},\varepsilon_{F'}}$ epi-converge to
$\mathcal{L}_{0,0}$ for $\varepsilon_F \to 0$ and $\varepsilon_{F'} \to 0$. In particular,
$\noisy{\Psi}_k(1) \to \Psi_k(1)$ for $\varepsilon_F \to 0$ and $\varepsilon_{F'} \to 0$ in case of a fixed
$x_k \in \R^n$.
\end{prop}
\begin{proof}
We begin by showing the first inequality and consider $d_\varepsilon \to d$. W.l.o.g.\ we assume that
$\liminf_{\varepsilon_{F},\varepsilon_{F'} \to 0} \mathcal{L}_{\varepsilon_{F},\varepsilon_{F'}}(d_\varepsilon)  < C$
for some $C > 0$, which implies that there is a subsequence (denoted by $d_\varepsilon$ as well) such that
$\normLP{d_\varepsilon} \le \Delta$ for all elements $d_\varepsilon$. The continuity of the norm $\normLP{\cdot}$
yields $0 = i_{\normLP{d_\varepsilon} \le \Delta} = i_{\normLP{d} \le \Delta}$. Moreover,
$\noisy{F}(x_k) \to F(x_k)$ and $\noisy{F}'(x_k) \to F'(x_k)$ for $\varepsilon_F \to 0$ and $\varepsilon_{F'} \to 0$
and thus the continuity of $\omega$ gives $\noisy{\ell}_k(d_\varepsilon) \to \ell_k(d)$ and in turn
the first inequality.

We continue with the second inequality and consider the constant sequence $d_\varepsilon \coloneqq d$.
Then $i_{\normLP{d_\varepsilon} \le \Delta} = i_{\normLP{d} \le \Delta}$ and
$\noisy{F}(x_k) \to F(x_k)$ and $\noisy{F}'(x_k) \to F'(x_k)$
for $\varepsilon_F \to 0$ and $\varepsilon_{F'} \to 0$. Again, the continuity of $\omega$ gives
$\noisy{\ell}_k(d_\varepsilon) \to \ell_k(d)$ and in turn the second inequality.

The functionals $\mathcal{L}_{\varepsilon_F,\varepsilon_{F'}}$ always admit a minimizer because
the feasible set $\{d\,|\,\normLP{d} \leq \Delta\}$, on which  $\mathcal{L}_{\varepsilon_F,\varepsilon_{F'}}$ is finite,
is compact, a standard argument yields that all accumulation points of a sequence of minimizers of the
functionals
$\mathcal{L}_{\varepsilon_F,\varepsilon_{F'}}$ minimize the limit functional $\mathcal{L}_{0,0}$.
\end{proof}
Consequently, if we drive $\noisy{\Psi}_k$ to zero over the iterations, we have an
upper bound on $\Psi_k$, defining a critical region (sublevel set) into which the iterates converge.
\begin{lem}
  \label{lem:combined_lipschitz}
  Under \cref{ass:F_lipschitz} it holds for all $x, d \in \R^n$ that
  \begin{equation*}
    \|F(x + d) - F(x) - F'(x) d\| \leq \LipFF \|d\|^2.
  \end{equation*}
\end{lem}
\begin{proof}
This follows directly from \cref{ass:F_lipschitz} and the mean value theorem.
\end{proof}

\begin{lem}
  \label{lem:exact_model_trial_diff}
  Under \cref{ass:omega_lipschitz,ass:F_lipschitz,ass:hess_bounded}, it holds that
  \begin{equation}
    | \noisy{\phi}(x_k + d_k) - \noisy{q}_k(d_k) | \
    \leq \FactConst + \FactLin \|d_k\| + \FactQuad \|d_k\|^2,
  \end{equation}
  where
  $\FactConst \define  2 \LipOmega \epsilon_{F}$,
  $\FactLin \define \LipOmega \epsilon_{F'}$, and
  $\FactQuad \define \LipOmega \LipFF + \onehalf \beta$.
\end{lem}
\begin{proof}
  The assumed Lipschitz continuity of $\omega$, $F$, $F'$, the representations $\noisy{F} = F + \delta_F$,
  $\noisy{F}' + \delta_{F'}$, \cref{lem:combined_lipschitz}, and the bounds on $\delta_F$, $\delta_{F'}$
  yield the claim with elementary computations.
\end{proof}
\begin{rem}
  If \cref{ass:omega_lipschitz,ass:F_lipschitz,ass:hess_bounded} hold
  in the noiseless case, the achievable reduction is related to the
  trial value via
  $|q_k(d_k) - \phi(x_k + d_k)| \leq M \|d_k\|^2$
  for some $M > 0$.  This is no longer the case in the
  noisy model.
\end{rem}
Several of the following results are due to~\cite{sleqp_convergence}
and are largely unaffected by moving from the noiseless to the noisy
regime. We refer to their counterparts in~\cite{sleqp_convergence}
and prove them in the appendix.  We begin by establishing that the
linearized model $\noisy{\ell}$ is still Lipschitz-continuous, albeit
with a Lipschitz-constant affected by the noise level $\epsilon_{F'}$:
\begin{lem}
  \label{lem:linear_zero_convexity}
  Under \cref{ass:omega_lipschitz,ass:F_lipschitz} it holds
  for all $d \in \R^n$ that
  \begin{equation*}
    |\noisy{\ell}_k(d) - \noisy{\ell}_k(0)| \leq \LipEll \normLP{d},
  \end{equation*}
  where $\LipEll \define \gamma \LipOmega (\LipFF + \epsilon_{F'})$.
\end{lem}
\begin{proof}
  The assumed Lipschitz continuity of $\omega$, $F'$, the representation $\noisy{F}' + \delta_{F'}$,
  \cref{lem:combined_lipschitz}. and the bound on $\delta_{F'}$ and the bound $\|\cdot\| \le \gamma \normLP{\cdot}$
  yield the claim with elementary computations.
\end{proof}
We proceed to examine the reduction according to the
partially linearized model as a function of the size of
an improvement step. The following result establishes that
the reduction is well behaved in the step size in following sense: 
the model reduction that is achieved for
a reduced step size is bounded from below by the model reduction
achieved without step reduction multiplied by the step reduction.
\begin{lem}
  \label{lem:linearized_convexity}
  It holds for all $\alpha \in [0, 1]$
  that
  \begin{equation*}
    \noisy{\phi}(x_k) - \noisy{\ell}_k(\alpha d) \geq \alpha [ \noisy{\phi}(x_k)  - \noisy{\ell}_k(d)].
  \end{equation*}
\end{lem}
\begin{proof}
  This follows directly from the fact that $\noisy{\ell}_k$ is convex,
  where we note that $\noisy{\ell}_k(0) = \noisy{\phi}(x_k)$ holds for $d = 0$.
\end{proof}
Next, we establish that the criticality $\noisy{\Psi}_k(\Delta)$ for a given
trust-region radius $\Delta > 0$ is bounded below by $\noisy{\Psi}_k(1)$ 
multiplied by $\Delta$ if the latter is less than one. This holds in
particular during the computation of the LP step in \cref{alg:noisy_sleqp}.
The proof requires the relationship established in \cref{lem:linearized_convexity}.
\begin{lem}[Lemma 3.2 in~\cite{sleqp_convergence}]
  \label{lem:critical_normalization}
  It holds for any $\Delta > 0$
  that
  \begin{equation*}
    \noisy{\Psi}_k(\Delta) \geq \min(\Delta, 1) \noisy{\Psi}_k(1).
  \end{equation*}
\end{lem}
\proofAppendix
The next result states that when progress is possible with respect to
the criticality $\noisy{\Psi}_k(1)$, the LP step either lies
on the trust-region boundary or has a norm proportional to
$\noisy{\Psi}_k(1)$.
\begin{lem}[Lemma 3.3 in~\cite{sleqp_convergence}]
  \label{lem:lp_step_criticality}
  If \cref{ass:omega_lipschitz,ass:F_lipschitz} hold and that
  ${\noisy{\Psi}_k(1) \neq 0}$. Let $d_{\Delta}$ be a minimizer
  achieving $\noisy{\Psi}_k(\Delta)$ for some $\Delta > 0$. Then it
  follows that
  \begin{equation*}
    \normLP{d_{\Delta}} \geq \min \left( \Delta, \frac{\noisy{\Psi}_k(1)}{\LipEll} \right).
  \end{equation*}
\end{lem}
\proofAppendix
We are now ready to examine the step $d_k$ computed by
\cref{alg:noisy_sleqp} with respect to the reduction achieved by the
model $\noisy{q}_k$. Specifically, if progress can be made with
respect to the criticality $\noisy{\Psi}_k(1)$, then we can expect a
positive reduction in $\noisy{q}_k$. We use this result to prove
that the objective $\noisy{\phi}$ decreases as well as long as
$\noisy{\Psi}_k(1)$ is sufficiently large.
\begin{lem}[Lemma 3.4 in~\cite{sleqp_convergence}]
  \label{lem:model_decrease}
  Under \cref{ass:omega_lipschitz,ass:F_lipschitz}, the
  model decrease satisfies
  \begin{equation*}
    \noisy{\phi}(x_k) - \noisy{q}_k(d_k) \geq
    \noisy{\phi}(x_k) - \noisy{q}_k(d^C_k) \geq
    \eta \alpha_k \noisy{\Psi}_k(\trustRadiusLP_k) \geq
    \eta \alpha_k \min(\trustRadiusLP_k, 1) \noisy{\Psi}_k(1).
  \end{equation*}
\end{lem}
\proofAppendix
The following technical lemma shows that if
$\noisy{\Psi}_k(1) \neq 0$, then $\normLP{d^{C}_{k}}$ is bounded
below, which we will need to ensure that the updated trust region
radii do not collapse while progress in the objective can still be
made.
\begin{lem}[Lemma 3.6 in~\cite{sleqp_convergence}]
  \label{lem:cauchy_step_size_bound}
  Under \cref{ass:omega_lipschitz,ass:F_lipschitz,ass:hess_bounded} it holds that
  \begin{equation*}
    \begin{aligned}
      \alpha_k \trustRadiusLP_k & \geq \normLP{d^{C}_{k}} \\
                                & \geq
                                  \min
                                  \left(
                                  \frac{\trustRadius_k}{\gamma},
                                  \trustRadiusLP_k,
                                  \frac{\noisy{\Psi}_k(1)}{\LipEll},
                                  \min \left( 1, \frac{1}{\trustRadiusLP_k} \right)
                                  \frac{2(1 - \eta)\tau \noisy{\Psi}_k(1)}{\beta \gamma^2}
                                  \right).
    \end{aligned}
  \end{equation*}
\end{lem}
\proofAppendix

\subsection{Lower Bounds on the Trust-region Radii}
\label{sec:lb_trust_region_radius}
We are now able to state a key result that provides lower bounds on
both the trust-region radius $\trustRadius$ for the quadratic model
and the LP trust-region radius $\trustRadiusLP$. It ensures that the
algorithm does not stall while progress can be made with respect to the 
noisy criticality $\noisy{\Psi}_k$. 
The proof strategy follows Lemma~3.7 in~\cite{sleqp_convergence} for the noiseless case.
In order to compensate for the noise, we need
to assume a sufficiently large stabilization 
parameter $\ratioParam$, which
in turn depends on the constants introduced by the noise. 
\begin{lem}
  \label{lem:tr_down}
  Consider an application of \cref{alg:noisy_sleqp} to the noisy
  variant of problem~\eqref{eq:convex_problem}. Suppose that
  \cref{ass:omega_lipschitz,ass:F_lipschitz,ass:hess_bounded} hold,
  $\noisy{\Psi}_k(1) \geq \delta > 0$ for all $k$, and that
  \begin{gather}
    \label{eq:required_stabilization}
    \ratioParam \geq \ratioParamOpt \define \frac{\FactConst + \FactLin}{1 - \rho_u},
  \end{gather}
  with $\FactConst$, $\FactLin$ from
  \cref{lem:exact_model_trial_diff}. Then it follows that
  \begin{equation*}
    \trustRadius_k \geq \Delta_{\min} \quad \mathrm{and}
    \quad \alpha_k \trustRadiusLP_k \geq \frac{\Delta_{\min}}{\gamma},
  \end{equation*}
  where $\Delta_{\min} = \min(A, B \delta)$ with
  \begin{equation*}
    \begin{aligned}
      A \define &
                  \min
                  \left(
                  \theta \gamma^{2},
                  \trustRadius_0,
                  \trustRadiusLP_0 \gamma,
                  \frac{\gamma}{\trustRadiusLPMax}
                  \right), \textrm{ and} \\
      B \define &
                  \min \left(
                  \frac{(1 - \rho_u) \eta}{\gamma \FactQuad \trustRadiusLPMax}
                  \min( \theta^{2}, \kappa_{l}^{2} ),
                  \frac{\gamma}{\LipEll},
                  \frac{2 (1 - \eta) \tau}{\beta \gamma \trustRadiusLPMax}
                  \right).
    \end{aligned}
  \end{equation*}
\end{lem}

\begin{proof}
  Using $\noisy{\Psi}_k(1) \ge \delta$, the bound in \cref{lem:cauchy_step_size_bound} becomes
  \begin{equation*}
    \begin{aligned}
      \normLP{d^{C}_{k}}
      \geq
      &
      \min
      \left(
        \frac{\trustRadius_k}{\gamma},
        \trustRadiusLP_k,
        \frac{\noisy{\Psi}_k(1)}{\LipEll},
        \min \left( 1, \frac{1}{\trustRadiusLP_k} \right)
        \frac{2(1 - \eta)\tau \noisy{\Psi}_k(1)}{\beta \gamma^2}
      \right) \\
      \geq
      &
      \min
      \left(
        \frac{\trustRadius_k}{\gamma},
        \trustRadiusLP_k,
        \frac{\delta}{\LipEll},
        \frac{2(1 - \eta)\tau \delta}{\trustRadiusLPMax \beta \gamma^2}
      \right) =
      \min
      \left(
        \frac{\trustRadius_k}{\gamma},
        \trustRadiusLP_k,
        \Delta_{\crit}
      \right),
    \end{aligned}
  \end{equation*}
  where
  \begin{equation*}
    \Delta_{\crit} \define \min \left( \frac{\delta}{\LipEll},
      \frac{2(1 - \eta)\tau \delta}{\trustRadiusLPMax \beta \gamma^2}  \right).
  \end{equation*}
  If a step is accepted in the $k$-th iteration (that is $\noisy{\rho}_k \geq \rho_u$),
  it follows that
  \begin{equation}
    \label{eq:lp_tr_bound_after_success}
    \begin{aligned}
      \trustRadiusLP_{k + 1} \geq
      \min
      \left(
        \frac{\trustRadius_k}{\gamma},
        \trustRadiusLP_k,
        \Delta_{\crit}
      \right).
    \end{aligned}
  \end{equation}
  If, on the other hand, the step is rejected, we can deduce
 the inequalities
  \begin{equation*}
      1 - \rho_u
      < \: 1 - \noisy{\rho}_k
      = \: \frac{\noisy{q}_k(d_k) -\noisy{\phi}(x_k + d_k)}{ \noisy{\phi}(x_k) - \noisy{q}_k(d_k) + \ratioParam}
      \leq \: \frac{\FactConst + \FactLin \|d_k\| + \FactQuad \|d_k\|^2}{ \eta \alpha_k \min \left( \trustRadiusLP_k, 1 \right) \noisy{\Psi}_k(1) + \ratioParam}
  \end{equation*}
  from \cref{lem:exact_model_trial_diff,lem:model_decrease}.
  Based on the bounds $\noisy{\Psi}_k(1) \ge \delta$, $\trustRadiusLP_k \le \trustRadiusLP_{\max}$, and
  $\normLP{d_k^C} = \normLP{\alpha_k d_k^{\LP}} \le \alpha_k \trustRadiusLP_k$,
  we can estimate the denominator via
  \begin{equation*}
    \begin{aligned}
      \: \eta \alpha_k \min \left( \trustRadiusLP_k, 1 \right) \noisy{\Psi}_k(1) + \ratioParam
      \geq & \: \eta \alpha_k \trustRadiusLP_k \min \left( 1, \frac{1}{\trustRadiusLP_k} \right) \delta + \ratioParam \\
      \geq & \: \eta \normLP{d^{C}_k} \frac{\delta}{\trustRadiusLPMax} + \ratioParam\\
      \underset{\eqref{eq:lp_tr_bound_after_success}}\geq & \: \eta \min \left( \frac{\trustRadius_k}{\gamma}, \trustRadiusLP_k, \Delta_{\crit} \right) \frac{\delta}{\trustRadiusLPMax}  + \ratioParam.
    \end{aligned}
  \end{equation*}
  Consequently, we obtain the relationship
  \begin{equation*}
    (1 - \rho_u) \eta \min \left( \frac{\trustRadius_k}{\gamma}, \trustRadiusLP_k, \Delta_{\crit} \right) \frac{\delta}{\trustRadiusLPMax}
    + (1 - \rho_u) \ratioParam \leq \FactConst + \FactLin \|d_k\| + \FactQuad \|d_k\|^2.
  \end{equation*}
  To finish the proof, we distinguish two cases with respect to $\|d_k\|$:
  \begin{enumerate}
  \item
    If $\|d_k\| \geq 1$, it follows that $\theta \normLP{d_k} \geq \theta \gamma \|d_k\| \geq \theta \gamma$,
    which implies that $\trustRadiusLP_{k + 1} \geq \min(\theta \gamma, \trustRadiusLP_{k})$.
  \item
    If $\|d_k\| < 1$, our lower bound on $\ratioParam$ implies
    \begin{equation}
      \label{eq:step_norm_bound}
      \frac{1 - \rho_u}{\FactQuad}
      \eta
      \min
      \left( \frac{\trustRadius_k}{\gamma}, \trustRadiusLP_k, \Delta_{\crit} \right)
      \frac{\delta}{\trustRadiusLPMax}
      \leq \|d_k\|^2.
    \end{equation}
    Recall that the next trust region radius $\trustRadiusLP_{k + 1}$ has a value of
    at least $\theta \normLP{d_k}$, which implies that
    \begin{equation*}
      \begin{aligned}
        \left( \trustRadiusLP_{k + 1} \right)^2
        \geq & \: \theta^2 \normLP{d_k}^2 
        \geq \: \theta^2 / \gamma^{2} \|d_k\|^2 \\
        \geq & \: \frac{\theta^2 (1 - \rho_u) \eta \delta}{\gamma^2 \FactQuad \trustRadiusLPMax}
               \min \left( \frac{\trustRadius_k}{\gamma}, \trustRadiusLP_k, \Delta_{\crit} \right) \\
        \geq & \: \min \left( \frac{\theta^2 (1 - \rho_u) \eta \delta}{\gamma^2 \FactQuad \trustRadiusLPMax} ,
               \frac{\trustRadius_k}{\gamma}, \trustRadiusLP_k, \Delta_{\crit} \right)^2.
      \end{aligned}
    \end{equation*}

  \end{enumerate}
  We combine both cases by taking their minimum, resulting in
  \begin{equation}
    \label{eq:lp_tr_bound_after_fail}
    \begin{aligned}
      \trustRadiusLP_{k + 1} \geq
      \min \left( \theta \gamma, \frac{\theta^2 (1 - \rho_u) \eta \delta}{\gamma^2 \FactQuad \trustRadiusLPMax} ,
      \frac{\trustRadius_k}{\gamma}, \trustRadiusLP_k, \Delta_{\crit} \right).
    \end{aligned}
  \end{equation}
  This lower bound on $\trustRadiusLP_{k + 1}$ dominates
  the previously shown lower bound~\eqref{eq:lp_tr_bound_after_success}
  for accepted iterates. In order to derive a uniform
  lower bound on $\trustRadiusLP_k$ (that is independent of $k$),  
  we may assume the worst case, i.e.~all steps
  are rejected, and resort to only~\eqref{eq:lp_tr_bound_after_fail}.

  Regarding the trust-region radius $\trustRadius_k$ for the quadratic model,
  we can follow a similar chain of reasoning as for $\trustRadius^{\text{LP}}_k$.  The
  radius is only decreased when the reduction ratio $\noisy{\rho}_k$
  is less than or equal to $\rho_s$, in which case it follows that
  $\trustRadius_{k + 1} \geq \kappa_l \|d_k\|$.  We can use the
  inequalities that lead to~\eqref{eq:lp_tr_bound_after_fail} to
  obtain that
  \begin{equation*}
      \frac{\trustRadius_{k + 1}}{\gamma}
      \geq \min
             \left( \theta \gamma, \frac{\trustRadius_k}{\gamma}, \trustRadiusLP_k, \Delta_{\crit},
               \frac{\delta \kappa_l^2 (1 - \rho_u) \eta}{\gamma^2 \FactQuad \trustRadiusLPMax}
             \right).
  \end{equation*}
  We can combine these estimates to obtain the lower bound
  \begin{equation*}
    \begin{aligned}
      \min \left( \frac{\trustRadius_{k+1}}{\gamma}, \trustRadiusLP_{k+1} \right) \geq
      \min
      \left( \theta \gamma, \frac{\trustRadius_k}{\gamma}, \trustRadiusLP_k, \Delta_{\crit},
      \frac{\theta^2 (1 - \rho_u) \eta \delta}{\gamma^2 \FactQuad \trustRadiusLPMax},
      \frac{\delta \kappa_l^2 (1 - \rho_u) \eta}{\gamma^2 \FactQuad \trustRadiusLPMax}
      \right).
    \end{aligned}
  \end{equation*}
  Starting from some $k$, we apply the inequality above recursively
  while decrementing $k$ and arrive at
  \begin{equation*}
    \begin{aligned}
      \min \left( \frac{\trustRadius_k}{\gamma}, \trustRadiusLP_k \right) \geq
      \min
      \left( \vphantom{\frac{\theta^2}{\trustRadiusLPMax}} \right. &
                                                                     \theta \gamma,
                                                                     \frac{\trustRadius_0}{\gamma}, \trustRadiusLP_0, \Delta_{\crit}, \\
                                                                   & \frac{\theta^2 (1 - \rho_u) \eta \delta}{\gamma^2 \FactQuad \trustRadiusLPMax},
                                                                     \frac{\delta \kappa_l^2 (1 - \rho_u) \eta}{\gamma^2 \FactQuad \trustRadiusLPMax}
                                                                     \left. \vphantom{\frac{\theta^2}{\trustRadiusLPMax}} \right)
                                                                     \enifed \Delta_{\low}.
    \end{aligned}
  \end{equation*}
  It must therefore hold that $\Delta_k \geq \gamma \Delta_{\low}$.
  Moreover, it holds that
  \begin{equation*}
    \begin{aligned}
      \alpha_k \trustRadiusLP_k \geq \normLP{d^{C}_k}
      \geq \min \left( \frac{\trustRadius_k}{\gamma},
      \trustRadiusLP_k, \Delta_{\crit} \right)
      \geq \min \left( \Delta_{\low}, \Delta_{\crit} \right)
       = \Delta_{\low}.
    \end{aligned}
  \end{equation*}
  The result follows from grouping the terms in $\gamma \Delta_{\low}$,
  according to whether or not they contain $\delta$.
\end{proof}
\subsection{Global Convergence Theorem}\label{sec:global_convergence_theorem}
We are now ready to establish the convergence of
\cref{alg:noisy_sleqp}. In order to simplify the proof of the main
theorem, we handle the special case in which
\cref{alg:noisy_sleqp} converges in a finite number of steps
separately.
\begin{lem}[Corollary 3.8 in~\cite{sleqp_convergence}]
  \label{lem:finite_successful_convergence}
  Consider an application of \cref{alg:noisy_sleqp} to the noisy
  variant of problem~\eqref{eq:convex_problem}. Suppose that
  \cref{ass:omega_lipschitz,ass:F_lipschitz,ass:hess_bounded}  
  and \eqref{eq:required_stabilization} hold.
  If there are finitely many successful iterations
  (that is $\noisy{\rho}_k \geq \rho_u$)
  during the execution of \cref{alg:noisy_sleqp}, then it holds that
  \begin{equation*}
    x_k = x^{*} \textrm{ and } \noisy{\Psi}_k(1) = 0
  \end{equation*}
  for all sufficiently large $k$.
\end{lem}
\proofAppendix

The following convergence theorem states that when the objective
of \eqref{eq:convex_problem} is bounded below, an application of
\cref{alg:noisy_sleqp} will produce one of two possible mutually
exclusive outcomes: the algorithm may stop at a critical point after
a finite number of iterations as described in
\cref{lem:finite_successful_convergence} or, alternatively,
\cref{alg:noisy_sleqp} visits a critical region infinitely
often. In terms of the functions $\omega$, $\noisy{F}$, and
$\noisy{F}'$, the \emph{critical region} is defined as
\begin{equation*}
  C(\delta) \define
  \left\{ x \middle| \noisy{\phi}(x)
  - \min_{\normLP{d} \leq 1}
  \omega(\noisy{F}(x) + \noisy{F'}(x) d)
  \leq \delta
  \right\}.
\end{equation*}
An iterate $x_k$ produced during the execution of
\cref{alg:noisy_sleqp} is contained in $C(\delta)$ if and only if
$\noisy{\Psi}_k(1) \leq \delta$.
What is more,
\cref{prop:noiseless_convergence} establishes that $\noisy{\Psi}_k(1)$
tends to $\Psi_k(1)$ as the errors $\epsilon_{F}$ and $\epsilon_{F'}$
approach zero. These results therefore suggest that the iterate is
\emph{close} to being optimal in the sense
of~\cref{lem:general_converge}.
\begin{thm}
  \label{thm:global_convergence}
  Consider an application of \cref{alg:noisy_sleqp} to the noisy
  variant of problem~\eqref{eq:convex_problem}. Suppose that
  \cref{ass:omega_lipschitz,ass:F_lipschitz,ass:hess_bounded}  
  and \eqref{eq:required_stabilization} hold. Then either
  \begin{equation*}
    \noisy{\Psi}_k(1) = 0 \: \: \textrm{for some} \: \: k \geq 0,
  \end{equation*}
  or
  \begin{equation*}
    \lim_{k \to \infty} \noisy{\phi}(x_k) = -\infty,
  \end{equation*}
  or there are infinitely many $k \in \N$ such that $x_k \in C(\delta_{\max})$, where
  \begin{equation*}
    \delta_{\max} \define
    \max
    \left(
    \sqrt{\frac{\ratioParam ( 1 -\rho_u ) \gamma \trustRadiusLPMax}
      {\rho_u \eta B}},
    \frac{\ratioParam ( 1 - \rho_u) \gamma \trustRadiusLPMax}
         {\rho_u \eta A}
         \right)
  \end{equation*}
  is given in terms of the constants $A, B$ from \cref{lem:tr_down}.
\end{thm}
\begin{proof}
  If there are only finitely many accepted steps, the
  result follows from
  \cref{lem:finite_successful_convergence}, yielding the
  first possibility. Otherwise, we can assume that during the
  algorithm, an infinite number of accepted steps occurs. If
  ${\noisy{\phi}(x_k)}$ tends to $-\infty$, the second possibility
  occurs, so we can assume in the following that $\noisy{\phi}(x_k)$
  (and hence $\phi(x_k)$) is bounded below.

  Let $\mathcal{K}$ be the sequence of accepted steps, \ie
  consisting of those $k$ where $x_{k+1} \neq x_{k}$. Clearly, if
  $\liminf_{k \to \infty} \noisy{\Psi}_k(1) = 0$, then the result
  follows. So we can assume that there exists a $\delta > 0$ such that
  $\noisy{\Psi}_k(1) \geq \delta$ for all $k \geq k_0$.  The claim
  stating that the region $C(\delta_{\max})$ is visited infinitely
  often is tantamount to ensuring that $\delta \leq \delta_{\max}$,
  which will be the aim of the remainder of this proof.  For each
  $k \in \mathcal{K}, k \geq k_0$ we have that
  \begin{equation*}
    \frac{\noisy{\phi}(x_k) - \noisy{\phi}(x_{k + 1}) + \ratioParam}{\noisy{\phi}(x_k) - \noisy{q}_k(d_k)+ \ratioParam} \geq \rho_u > 0.
  \end{equation*}
  We deduce using \cref{lem:model_decrease} that
  \begin{equation*}
    \begin{aligned}
      \noisy{\phi}(x_k) - \noisy{\phi}(x_{k + 1})
      \geq & \: \rho_u
      \left[ \noisy{\phi}(x_k) - \noisy{q}_k(d_k)  \right] +
      (\rho_u - 1) \ratioParam \\
      \geq & \: \rho_u
      \left[ \eta \alpha_k \min(\trustRadiusLP_k, 1)
        \noisy{\Psi}_k(1) \right] +
      (\rho_u - 1) \ratioParam.
    \end{aligned}
  \end{equation*}
  It follows that
  \begin{equation*}
    \begin{aligned}
      \noisy{\phi}(x_k) - \noisy{\phi}(x_{k + 1})
      \geq & \: \rho_u
      \left[ \eta \alpha_k \trustRadiusLP_k \min
        \left( \frac{1}{\trustRadiusLP_k}, 1 \right)
        \noisy{\Psi}_k(1) \right] +
      (\rho_u - 1) \ratioParam\\
      \geq & \: \rho_u
      \left[ \eta \alpha_k \trustRadiusLP_k \frac{1}{\trustRadiusLPMax}
        \noisy{\Psi}_k(1) \right] +
      (\rho_u - 1) \ratioParam.
    \end{aligned}
  \end{equation*}
  We can now apply \cref{lem:tr_down} to bound
  $\alpha_k \trustRadiusLP_k$ below based on $\Delta_{\min}$
  and the constants $A$ and $B$:
  \begin{equation*}
    \begin{aligned}
      \noisy{\phi}(x_k) - \noisy{\phi}(x_{k + 1})
      \geq & \: \rho_u
      \left[ \eta \frac{\Delta_{\min}}{\gamma \trustRadiusLPMax}
        \noisy{\Psi}_k(1) \right] +
      (\rho_u - 1) \ratioParam\\
      = & \: \rho_u
      \left[ \eta \frac{\min(A, B \delta)}{\gamma \trustRadiusLPMax}
        \noisy{\Psi}_k(1) \right] +
      (\rho_u - 1) \ratioParam\\
      \geq & \: \rho_u
      \left[ \eta \frac{\min(A, B \delta)}{\gamma \trustRadiusLPMax}
        \delta \right] +
      (\rho_u - 1) \ratioParam.
    \end{aligned}
  \end{equation*}
  Let us assume towards a contradiction that $\delta > \delta_{\max}$.
  We distinguish two cases with respect to the minimum $\min(A,B \delta)$:
  \begin{enumerate}
  \item
    The minimum is attained at $A$, implying that
    \begin{equation*}
      \begin{aligned}
        \noisy{\phi}(x_k) - \noisy{\phi}(x_{k + 1})
        \geq & \: \rho_u
        \left[ \eta \frac{A}{\gamma \trustRadiusLPMax}
          \delta \right] +
        (\rho_u - 1) \ratioParam\\
        \geq & \: \rho_u
        \left[ \eta \frac{A}{\gamma \trustRadiusLPMax}
          \delta_{\max} \right] +
        (\rho_u - 1) \ratioParam + C_1.
      \end{aligned}
    \end{equation*}
    for a constant $C_1 > 0$. Using the fact that
    \begin{equation*}
      \delta_{\max} \geq \frac{\ratioParam( 1 - \rho_u) \gamma
        \trustRadiusLPMax} {\rho_u \eta A},
    \end{equation*}
    this implies that
    $\noisy{\phi}(x_k) - \noisy{\phi}(x_{k + 1}) \geq C_1 > 0$.
  \item
    The minimum is attained at $B \delta$,
    implying that
    \begin{equation*}
      \begin{aligned}
        \noisy{\phi}(x_k) - \noisy{\phi}(x_{k + 1})
        \geq & \: \rho_u
        \left[ \eta \frac{B \delta^2}{\gamma \trustRadiusLPMax}
          \right] +
        (\rho_u - 1) \ratioParam\\
        \geq & \: \rho_u
        \left[ \eta \frac{B \delta_{\max}^2}{\gamma \trustRadiusLPMax}
          \right] +
        (\rho_u - 1) \ratioParam +
        C_2 \\
      \end{aligned}
    \end{equation*}
    for a constant $C_2 > 0$. We now use the fact that
    \begin{equation*}
      \begin{aligned}
        \delta_{\max} \geq \sqrt{\frac{\ratioParam ( 1 -\rho_u ) \gamma \trustRadiusLPMax}
          {\rho_u \eta B}}
      \end{aligned}
    \end{equation*}
    to deduce that
    $\noisy{\phi}(x_k) - \noisy{\phi}(x_{k + 1}) \geq C_2 > 0$.
  \end{enumerate}
  In either case $\noisy{\phi}$ decreases by $\min(C_1, C_2) > 0$
  from $x_k$ to $x_{k+1}$. Since this decrease is strictly
  positive and there are infinitely many accepted steps in the
  sequence $\mathcal{K}$, it follows that $\noisy{\phi}(x_k)$ tends to
  $-\infty$, which is a contradiction. It must therefore hold that
  $\delta \leq \delta_{\max}$ as desired.
\end{proof}

\paragraph{Interpretation of Theorem~\ref{thm:global_convergence}}

In theoretical terms, the result in
\cref{thm:global_convergence} is as expected: the size of the
critical region $C$ depends on the stabilization parameter $\ratioParam$,
If we increase $\ratioParam$, \cref{alg:noisy_sleqp} can tolerate
a larger amount of noise at the cost of a decreased accuracy with
respect to the criticality measure $\noisy{\Psi}$. Of course,
problem~\eqref{eq:convex_problem} can generally also be unbounded. The
remaining case, where $\noisy{\Psi}_k(1) = 0$ for some $k \in \N$, can
involve different scenarios. If $\Psi_k(1) = 0$ holds as well, the
iterate $x_k$ is a critical point of~\eqref{eq:convex_problem}, which
is the ideal situation. Otherwise, the noises $\delta_F(x_k)$ and
$\delta_{F'}(x_k)$ attain values such that $x_k$ \emph{appears to be
critical} in the noisy model. In case of an unconstrained version
of~\eqref{eq:nonlinear_problem}, this is tantamount to a non-zero
gradient that is canceled out by noise.

For a function $F$ that is not afflicted by noise, \ie satisfying
$\epsilon_{F} = \epsilon_{F'} = 0$, it holds that $M_{\epsilon} = 0$,
allowing us to set $\ratioParam = 0$, whereby we recover the original
algorithm discussed in~\cite{sleqp_convergence}.  If we apply
\cref{thm:global_convergence} in this situation, it follow from
$\ratioParam = 0$ that $\delta_{\max} = 0$, implying that the region
$C(\delta_{\max})$ contains precisely the points critical with respect
to $\noisy{\Phi}$, which itself coincides with $\Phi$ in this
particular case.  Thus, the original convergence
result~\cite[Theorem~3.9]{sleqp_convergence} follows
for~\cref{alg:noisy_sleqp} in the noiseless case.

We can also apply~\cref{alg:noisy_sleqp} to solve smooth unconstrained
nonlinear problems affected by noise. To this end, let $p = 1$ and 
$\omega(x) = x$ for $x \in \R$, which leads to the Lipschitz constant
$\LipOmega = 1$ so that \cref{thm:global_convergence} suggests
a stabilization of
\begin{equation*}
  \ratioParam \geq \frac{2 \epsilon_{F} + \epsilon_{F'}}{1 - \rho_u},
\end{equation*}
which is weaker than the stabilization of $2 \epsilon_{F} / (1 - \rho_s)$
analyzed in~\cite{noisy_trust_region}. The weaker estimate is due to
the estimate from \cref{lem:linearized_convexity} that uses the
convexity of $\omega$. Instead, the authors of~\cite{noisy_trust_region}
have a Lipschitz continuous derivative of the objective at hand.
Then the criticality measure satisfies
\begin{equation*}
  x \in C(\delta) \:\: \Longleftrightarrow \:\:
  \left(- \min_{\normLP{d} \leq 1} \noisy{F}'(x) d \right) \leq \delta,
\end{equation*}
which in turn is equivalent to $\|\noisy{F}'(x) \|$ being bounded
above by a constant.

\section{Numerical Experiments}
\label{sec:numerics}
In order to illustrate the performance and examine the behavior of
\cref{alg:noisy_sleqp}, we implement the algorithm in \texttt{Python}
(3.8.10), using \texttt{numpy}~\cite{numpy} (1.24.1),
\texttt{scipy}~\cite{scipy} (1.9.3)
(including \texttt{HiGHS}~\cite{HiGHS} (1.2.0) as LP solver),
and \texttt{Ipopt}~\cite{ipopt} (3.11.9) to solve the
respective subproblems.  We generally compare the performance of the
\emph{classical} algorithm (\ie \cref{alg:noisy_sleqp} with a
stabilization of $\ratioParam = 0$), with its \emph{stabilized}
counterpart (where $\ratioParam > 0$).

In terms of termination criteria, we first of all impose an iteration
limit, after which the algorithm terminates. Secondly, we monitor the
LP trust-region radius, $\trustRadiusLP$. If the radius contracts to a
value close to zero (\num{1e-10}), we see this as a failure of the
algorithm and let it terminate. Lastly, when the noise criticality
$\noisy{\Psi}_k(1)$ falls below the threshold of \num{1e-6}, we
terminate the algorithm, knowing that the current iterate is very
close to being optimal. We then examine the \emph{final} iterate
$x^{f}$, \ie the iterate $x^{k}$ in~\cref{alg:noisy_sleqp}
of the iteration at which the termination criterion
becomes satisfied. If not indicated otherwise, we choose the
parameters of~\cref{alg:noisy_sleqp} according to the values
in~\cref{table:parameters} and the stabilization parameter
$\ratioParamOpt$.
\begin{table}[ht]
  \centering
  \caption{Parameters used for numerical experiments.}
  \begin{tabular}[ht]{llS}
  \toprule
  \textbf{Symbol} & \textbf{Explanation} & {\textbf{Value}} \\
  \midrule
  $\trustRadiusLP_0$ & Initial LP trust-region radius & 1 \\
  $\trustRadiusLPMax$ & Maximum LP trust-region radius & 10 \\
  $\trustRadius_0$ & Initial trust-region radius & 1 \\
  $\rho_u$ & Step acceptance threshold & 0.1 \\
  $\rho_s$ & Threshold for increase of $\trustRadius$ & 0.5 \\
  $\kappa_l$ & Lower bound for adjustment of $\trustRadius$ after failed step & 0.1 \\
  $\kappa_u$ & Upper bound for adjustment of $\trustRadius$ after failed step & 0.8 \\
  $\theta$ & Lower bound for adjustment of $\trustRadiusLP$ after failed step & 0.5 \\
  $\eta$ & Factor of relative decrease for Cauchy step & 0.1 \\
  $\tau$ & Shortening factor for Cauchy line search & 0.5 \\
  \bottomrule
\end{tabular}

  \label{table:parameters}
\end{table}
In order to obtain inexact evaluations of a given function
$F : \R^{n} \to \R^{p}$ and its derivative, we inject noise by setting
\begin{equation}
  \label{eq:default_noise_model}
  \begin{aligned}
    &\delta_{F}(x) = X_{F} \in \R^{p}, \: X_{F} \sim \mathbb{B}_{n}(\epsilon_{F}) \text{, and } \\
    &\delta_{F'}(x) = X_{F'} \in \R^{p \times n}, \: X_{F'} \sim \mathbb{B}_{np}(\epsilon_{F'}),
  \end{aligned}
\end{equation}
where $\mathbb{B}_{n}(s)$ denotes the uniform distribution on the
$n$-dimensional ball centered at the origin with radius $s$. Sampling
randomly from these distributions at each point $x$ ensures that the
amount of noise is bounded according to the noise levels $\epsilon_{F}$
and $\epsilon_{F'}$ while being sufficiently unpredictable.

In \cref{sec:failure_classical}, we augment an
example of a quadratic test problem from
\cite{noisy_trust_region} with
a non-smooth term. We obtain qualitatively
similar results in this case.
Then we compare and visualize the different
behaviors of the unstabilized algorithm and
the stabilized algorithm for the Rosenbrock
test function in \cref{sec:rosenbrock}.
In \cref{sec:image_denoising}, we apply
the algorithm to an image reconstruction problem
with total variation regularization
and assess the impact of different choices of the
stabilization parameter.
Finally, in \cref{sec:constrained_optimization},
we apply the algorithm in a penalty method for
a small constrained optimization problem from
\texttt{CUTest}~\cite{cutest} as motivated in the introduction, which points to
future research directions.
\Cref{sec:failure_classical,sec:rosenbrock,sec:constrained_optimization} use
essentially the same type of a non-smooth objective function that includes
an $\ell^1$-penalty term and we provide the required estimates of its
Lipschitz constant in \cref{sec:estimations}.

\subsection{Failure of the Classical Algorithm}
\label{sec:failure_classical}

To illustrate the difference in performance between the classical
algorithm and~\cref{alg:noisy_sleqp}, we consider the case of
$\ell^{1}$-penalized optimization problems of the form
$x \mapsto f(x) + \lambda \|x\|_1$ with a smooth function $f : \R^{n} \to
\R$. It is clear that these problems are non-smooth due to the
presence of the $\|\cdot\|_1$ term, while being expressible as
problems of type~\eqref{eq:convex_problem} based on suitable choices
of $\omega$ and $F$. This problem class also enables us to minimize
$\noisy{\ell}_k$ and $\noisy{q}_k$ over the trust regions defined in
terms of $\trustRadiusLP$ and $\trustRadius$ by solving linear or
quadratic programs respectively. What is more, the only curvature
information in this problem class is due to $f$, enabling us to either
use the Hessian of $\noisy{f}$ (or any quasi-Newton approximation) to
obtain the matrices $B_k$. We specifically examine the case where $f$
is a quadratic of the form
\begin{equation*}
  f(x) = \onehalf \langle x, D x \rangle,
\end{equation*}
where $D$ is the matrix in $\R^{n \times n}$ for $n = 8$ given as
\begin{equation*}
  D = \diag (10^{-5}, 10^{-4.75}, 10^{-4.5}, \ldots, 10^{-3.25}),
\end{equation*}
taken from \cite{noisy_trust_region}, where this optimization problem
has been studied without an $\ell^1$-penalty. It is apparent that the
optimal solution of this instance of~\eqref{eq:convex_problem} is
$x^{\ast} = 0$.  We set the parameter $\lambda$ to \num{1e-2} while
injecting noise according to~\eqref{eq:default_noise_model} with noise levels
$\epsilon_{F} = \num{1e-1}$ and $\epsilon_{F'} = \num{1e-5}$, and
initialize~\cref{alg:noisy_sleqp} with the initial point
$x_0 = (\num{1000}, 0, \ldots, 0)^T$, limiting the number of
iterations to \num{50}, and performing quadratic steps based on the
true Hessian $D$.

We show an example of the difference in performance in
\Cref{figure:quad_convergence}, where the values of the reduction
ratio are clipped to $\pm 5$ in order to properly display the
results. We see that the classical algorithm performs dramatically
worse than its stabilized counterpart. Indeed, the classical algorithm
stalls almost immediately, due to the reduction ratio $\rho_k$
becoming unreliable.  Consequently, the LP trust region collapses, and
the classical algorithm makes no progress towards optimality.
Conversely, the addition of a stabilization yields an algorithm
rapidly approaching the optimum, both in terms of primal distance and
objective value while maintaining a reasonably large LP trust-region
radius. Similarly, noisy and noiseless criticality decrease rapidly throughout
the iterations of the stabilized algorithm.
Unfortunately, the criticality bound established
in~\Cref{thm:global_convergence} attains a value of
$\delta_{\max} \approx \num{6000}$, limiting its use in terms of the
criticality actually achieved throughout the iterations.

We would like to point out that the failure of the classical algorithm
is not guaranteed in this scenario: By running the experiment with
\num{100} different random seeds we found that the classical algorithm
stalls in about half (\num{45}) of the cases, while performing well in
the other half. Conversely, the stabilized algorithm consistently
performs well in all cases. Its characteristics are qualitatively
similar to the case that is depicted in
\Cref{figure:quad_convergence}. A key problem of the classical
algorithm is therefore its unreliability when applied to noisy
functions.

\begin{figure}[ht]
  \centering
  \begin{tikzpicture}
  \begin{groupplot}[
    group style={
      group size=2 by 3,
      group name=plots,
      horizontal sep=1.5cm,
      vertical sep=1.5cm},
    width=.5\textwidth]

    \newcommand{\PlotComparison}[1]{
      \addplot+[mark=none,very thick,ClassicalColor] table[col sep=semicolon,x=Iteration,y=Classical,Classical] {#1};
      \addplot+[mark=none,StabilizedColor] table[col sep=semicolon,x=Iteration,y=Stabilized,Stabilized] {#1};
    }
    \nextgroupplot[Iterations,
      legend columns=2,
      legend to name={QuadraticCommonLegend},
      label={Reduction ratio},
      ymin=-5,
      ymax=5,
      ytick={-4, -2, ..., 4},
      minor ytick={-5, -3, ..., 5}]  
    \PlotComparison{Data/Quadratic/ReductionRatio.csv}
    
    \addlegendentry{Classical algorithm}
    \addlegendentry{Stabilized algorithm}

    \nextgroupplot[Iterations,ymode=log]
    \PlotComparison{Data/Quadratic/OptDist.csv}

    \nextgroupplot[Iterations,ymode=log,ytick={1e-8,1e-5,1e-2,1e1}]
    \PlotComparison{Data/Quadratic/LPRadius.csv}

    \nextgroupplot[Iterations,ymode=log]
    \PlotComparison{Data/Quadratic/Objective.csv}

    \nextgroupplot[Iterations,ymin=0,ymax=0.04]
    \PlotComparison{Data/Quadratic/Criticality.csv}

    \nextgroupplot[Iterations,ymin=0,ymax=0.04]
    \PlotComparison{Data/Quadratic/NoisyCriticality.csv}
  \end{groupplot}
  \node[align=center,text width=20em,anchor=north] at ([yshift=-2.5mm]plots c1r1.south) {\subcaption{(Clipped) reduction ratio $\noisy{\rho}_k$}};
  \node[align=center,text width=20em,anchor=north] at ([yshift=-2.5mm]plots c2r1.south) {\subcaption{Distance $\|x_k - x^{\ast}\|$}};
  \node[align=center,text width=20em,anchor=north] at ([yshift=-2.5mm]plots c1r2.south) {\subcaption{LP trust region radius $\trustRadiusLP_k$}};
  \node[align=center,text width=20em,anchor=north] at ([yshift=-2.5mm]plots c2r2.south) {\subcaption{Objective $\omega(F(x_k))$}};
  \node[align=center,text width=20em,anchor=north] at ([yshift=-2.5mm]plots c1r3.south) {\subcaption{Criticality $\Psi_k(1)$}};
  \node[align=center,text width=20em,anchor=north] at ([yshift=-2.5mm]plots c2r3.south) {\subcaption{Noisy criticality $\noisy{\Psi}_k(1)$}};

  \path (plots c1r3.south east) -- node[midway,yshift=-1.5cm,anchor=north]{\ref{QuadraticCommonLegend}} (plots c2r3.south west);
\end{tikzpicture}
  \caption{
    Performance on an $\ell^{1}$-penalized quadratic problem over 50
    iterations of \cref{alg:noisy_sleqp} with noise levels of
    $\epsilon_F = \num{1e-1}$ and $\epsilon_{F'} = \num{1e-5}$.}
  \label{figure:quad_convergence}
\end{figure}

\subsection{A Variant of the Rosenbrock Problem}
\label{sec:rosenbrock}

Following the previous experiments conducted based on the quadratic
function, we go on to examine the performance on a variant of the
famous Rosenbrock function, given by
\begin{equation*}
  R(x, y) \define (a - x)^{2} + b(y - x^2)^2
\end{equation*}
with parameters of $a = \num{1}$, $b = \num{100}$. The Rosenbrock
function has a unique optimum at $(x^{\ast}, y^{\ast}) = (a, a^{2})$,
\ie at $(1, 1)$ for our choice of parameters. We modify the problem by
adding a penalty of $\lambda \|(x, y) - (x^{\ast}, y^{\ast})\|_1$ with
a value of $\lambda = \num{1e-1}$, yielding a problem of
type~\eqref{eq:convex_problem} having the same global optimum as
$R$. We show an example of the difference in performance between the
classical and stabilized algorithms in
\Cref{figure:rosenbrock_trajectories}. The figure shows the
trajectories generated by \Cref{alg:noisy_sleqp} with and without
stabilization starting at $(x^{0}, y^{0}) = (\num{-1.5}, \num{0})$,
injecting noise according to~\eqref{eq:default_noise_model} for
different values of $\epsilon_{F}$ and a fixed value of
$\epsilon_{F'} = \num{1e-5}$, performing quadratic steps
according to the true Hessian of $R$ with an iteration limit of $\num{50}$.

Examining the trajectories of the classical algorithm, shown in
\Cref{figure:rosenbrock_classical_trajectories}, we find that for
different values of $\epsilon_{F}$, the trajectories are initially
almost identical, until the algorithm stalls at points with a distances
to the optimum increasing with $\epsilon_{F}$. Conversely, the
trajectories of the stabilized algorithm, shown in
\Cref{figure:rosenbrock_stabilized_trajectories}, vary significantly
for different noise levels. However, the stabilization yields
trajectories leading significantly closer to the optimum than those of
the classical algorithm even for larger noise levels. This is
confirmed by the statistics shown in
\Cref{figure:rosenbrock_opt_distance}, displaying the distribution of
the distance to the optimum for various noise levels for \num{100}
different random seeds, demonstrating that the stabilized algorithm consistently
outperforms the classical one, in particular for larger noise levels.

\begin{figure}[hb]
  \centering
  \begin{subfigure}[t]{\textwidth}
    \includegraphics[width=\textwidth]{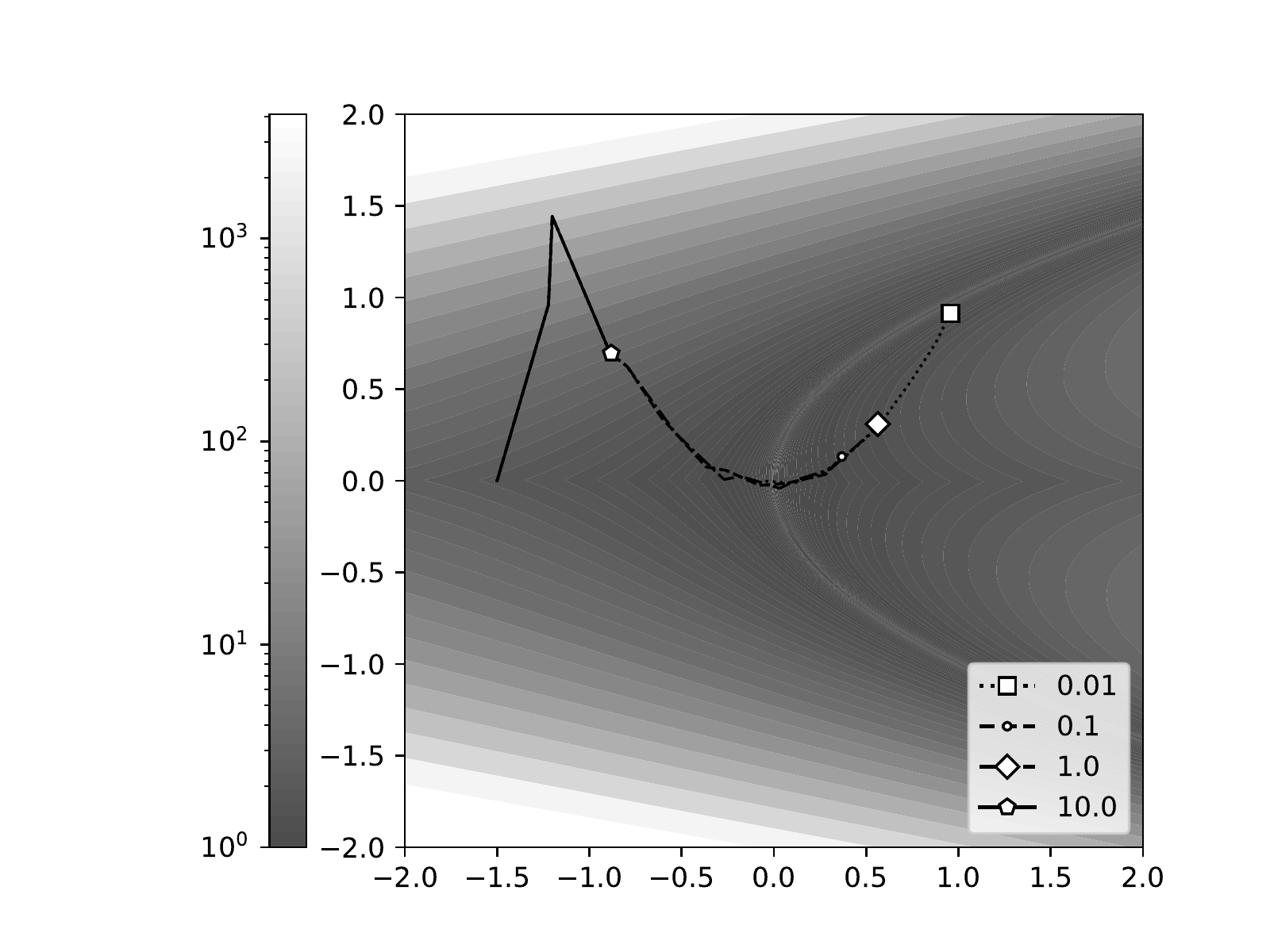}
    \caption{
      Trajectories of \cref{alg:noisy_sleqp}
      with $\ratioParam = 0$ for different noise levels.
    }
    \label{figure:rosenbrock_classical_trajectories}
  \end{subfigure}
  \\
  \begin{subfigure}[t]{\textwidth}
    \includegraphics[width=\textwidth]{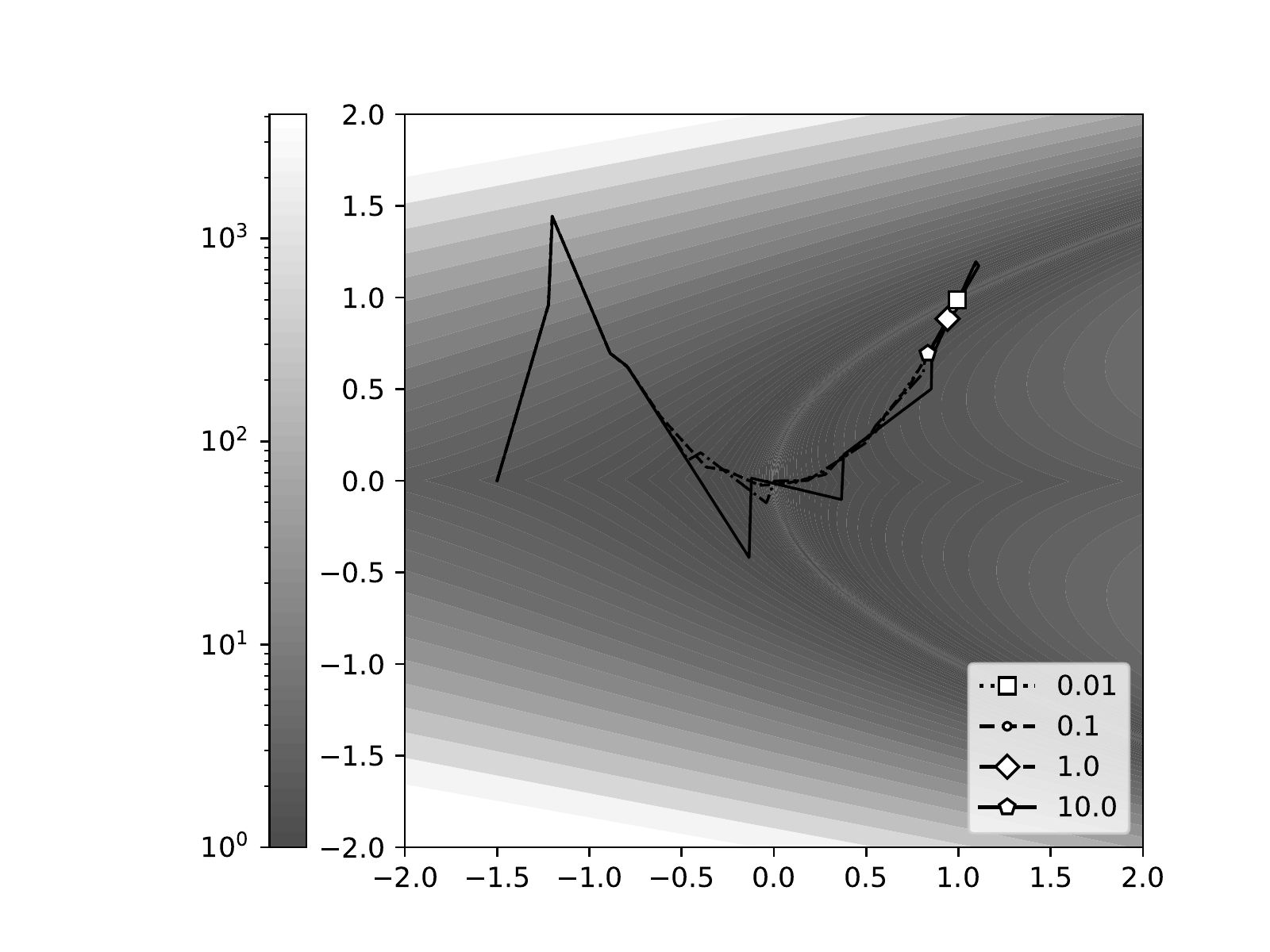}
    \caption{
      Trajectories of \cref{alg:noisy_sleqp} with $\ratioParam$ according
      to~\eqref{eq:required_stabilization} for different noise
      levels.
    }
    \label{figure:rosenbrock_stabilized_trajectories}
  \end{subfigure}
  \caption{
    Trajectories for the modified Rosenbrock problem plotted over the \emph{shifted} criticality
    $1 + \phi(x) - \min_{\normLP{d} \leq 1} \omega(F(x) + F'(x) d)$.
    The markers show the position of the final iterate.
  }
  \label{figure:rosenbrock_trajectories}
\end{figure}

\begin{figure}[hb]
  \centering
  \includegraphics{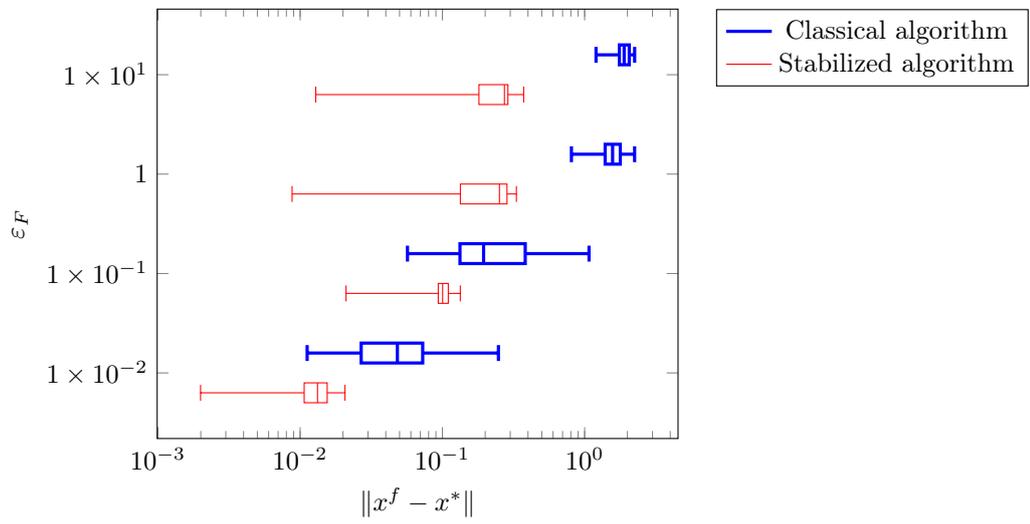}
  \caption{
    Distribution of the distance between the final iterate $x^{f}$ and noiseless optimum $x^{*}$
    for different values of the noise level
    $\epsilon_F$.}
  \label{figure:rosenbrock_opt_distance}
\end{figure}

\subsection{Image Reconstruction}
\label{sec:image_denoising}

Although this is not the focus of this article, we also provide a computational
example that has a meaningful problem size. Specifically, we consider an
artificial task of reconstructing an image under noisy observations.
That is we seek to recover a matrix $Y \in \R^{M \times N}$ with values normalized to
be in $[0, 1]$. In our setting, $Y$ is only
available in the form of noisy observations. Specifically, for an
input $X \in \R^{M \times N}$, the fidelity of $X$, given by the term
$\onehalf \normFrob{X - Y}^{2}$ and its derivative with respect to $X$
cannot be evaluated.
Instead, we have access to the map $X \mapsto \onehalf \normFrob{X - \noisy{Y}}^{2}$,
where $\noisy{Y}$ is a noisy version of $Y$, redrawn for each guess
$X$. We obtain the term $\noisy{Y}$ by sampling from a uniform distribution
\begin{equation*}
  \delta_{Y}(X) = Y_{F'} \in \R^{M \times N}, \: Y_{F'} \sim U_{M \times N}(-\imageNoise, \imageNoise),
\end{equation*}
and setting $\noisy{Y}$ to $Y + \delta_{Y}(X)$ clipped back to have
coefficients in $[0, 1]$. The amount of noise injected to the image is
in turn governed by the parameter $\imageNoise \geq 0$. This noise model
translates into noise injected into the evaluations
of $F$ and $F'$, which can be estimated in terms of $\imageNoise$, $M$,
and $N$ (see~\cref{sec:estimations}) while not conforming to the
noise model~\eqref{eq:default_noise_model}.
We also impose an anisotropic total variation (TV) regularization penalty,
defined as
\begin{equation*}
  \TV(X) \define \sum_{i = 1}^{M - 1} \sum_{j = 1}^{N} |X_{i +1, j} - X_{i, j}|
  + \sum_{i = 1}^{M} \sum_{j = 1}^{N - 1} |X_{i, j + 1} - X_{i, j}|,
\end{equation*}
to our objective, turning the problem non-smooth, balancing off
fidelity and regularity by a parameter $\lambda > 0$. The regularization
term can be expressed as $\|A_{M, N} X\|_1$ with a suitable matrix
$A_{M, N}$. Consequently, we can formulate the
reconstruction problem as problem of type~\eqref{eq:convex_problem},
consisting of a smooth term (the fidelity), and an
$\ell^{1}$-penalized linear function. Naturally, the regularization
does not suffer from any noise.

Based on a regularization parameter of $\lambda = \num{5e-3}$ we
reconstruct the image shown in \Cref{figure:denoising_original}. To avoid
having to solve large quadratic problems, we do not compute quadratic
steps and opt to instead increase the number of iterations to
\num{100} starting from $X_0 = 0$. As a baseline,
\Cref{figure:denoising_noiseless} shows the image when we apply
\Cref{alg:noisy_sleqp} to the original image (\ie setting
$\imageNoise = 0$).  The restored image closely resembles
the original one.

We proceed to study the effect of the value of $\ratioParam$ on the
quality of the reconstructed image. In principle, it must hold that
$\ratioParam \geq \ratioParamOpt$ in order for the criticality to
provably converge. It is however unclear whether setting $\ratioParam$
to $\ratioParamOpt$ yields the best results in practice. The large
value of $\delta_{\max}$ seen in \cref{sec:failure_classical} seems to
suggest that~\eqref{eq:required_stabilization} is rather
pessimistic. We therefore examine the performance of
\cref{alg:noisy_sleqp} for values of $\ratioParam$ not necessarily
satisfying the inequality.

To gauge performance, we record both the original and noisy objective
after the iterations.  The
results, shown in~\Cref{figure:denoising_performance}, demonstrate the
effect of $\ratioParam$: For small stabilization values,
\cref{alg:noisy_sleqp} stalls early on, as was the case in our
previous experiments. As we increase $\ratioParam$, there appears to
be an optimal choice or small region, where both the noisy and the
noiseless evaluation of the final objective are minimized.  This
effect is more pronounced for higher values of $\imageNoise$, where
the noiseless objective for $\ratioParam = 0$ is about 4 times as large
as that of the optimal choice of $\ratioParam$. It is
interesting to see that this sweet spot also shows in the noisy
objective, suggesting that noisy observations may be sufficient to
find it.  Lastly, as we increase $\ratioParam$ beyond the sweet spot,
the final objective increases sharply. This is likely due to the case
that the algorithm simply accepts too many steps, even when they are
in fact disadvantageous in terms of progressing towards an
optimum. Ultimately, for a sufficiently large value of $\ratioParam$,
all steps are accepted, which, as the final objective suggests, leads
to poor solutions.  The values of $\ratioParamOpt$ are given
by \largeNum{21990}, \largeNum{112140}, and \largeNum{229745} for
the respective noise levels, significantly exceeding the optimal values
and beyond the point, where all steps are accepted.

We also find that the objectives are consistent with the visual
appearance of the reconstructed images, shown in
\Cref{figure:denoising_results}: While setting $\ratioParam$ to zero
yields satisfactory results, even though a grainy appearance remains
for larger noise levels, a disproportionately large value of
$\ratioParam$ produces a distorted result with visible artifacts. For
our best guess of $\ratioParam$, the restored images do not suffer
from artifacts and closely resemble the original one even for larger
noise levels.

\begin{figure}[hb]
  \centering
  \begin{tikzpicture}
  \begin{groupplot}[
    group style={
      group size=1 by 3,
      group name=plots,
      xlabels at=edge bottom,
      vertical sep=1.5cm},
    width=.6\textwidth,
    xlabel={$\ratioParam$},
    ]

    \newcommand{\PlotNoiseLevel}[1]{
      \addplot+[mark=o,color=NoisyColor] table[col sep=semicolon,x=Stabilization,y=NoisyObjVal] {Data/Denoising/Results_#1.csv};
      \addplot+[mark=square,color=NoiselessColor] table[col sep=semicolon,x=Stabilization,y=ObjVal] {Data/Denoising/Results_#1.csv};
    }
    
    \nextgroupplot[
    legend to name={DenoisingCommonLegend},
    xmode=log
    ]

    \PlotNoiseLevel{0.01}

    \addlegendentry{$\omega(\noisy{F}(x))$}
    \addlegendentry{$\omega(F(x))$}
    
    \nextgroupplot[
    xmode=log
    ]

    \PlotNoiseLevel{0.05}

    \nextgroupplot[xmode=log]

    \PlotNoiseLevel{0.1}
  \end{groupplot}

  \node[align=center,inner sep=0,text width=20em,anchor=north] at (plots c1r1.outer south) {\subcaption{$\imageNoise = 0.01$,
  $\ratioParamOpt$~$\approx$~\largeNum{21990}}};
  \node[align=center,inner sep=0,text width=20em,anchor=north] at (plots c1r2.outer south) {\subcaption{$\imageNoise = 0.05$,
  $\ratioParamOpt$~$\approx$~\largeNum{112140}}};
  \node[align=center,inner sep=0,text width=20em,anchor=north] at (plots c1r3.outer south) {\subcaption{$\imageNoise = 0.1$,
  $\ratioParamOpt$~$\approx$~\largeNum{229745}}};

  \path (plots c1r1.north east) -- node[inner sep=0.,pos=0.0,xshift=.5cm,anchor=north west]{\ref{DenoisingCommonLegend}} (plots c1r1.south east);
\end{tikzpicture}
  \caption{
    Noiseless (red) and noisy evaluations (blue) of the objective values
    achieved by the final iterate of \cref{alg:noisy_sleqp} on an image reconstruction
    problem with different noise levels and stabilization parameters.}
  \label{figure:denoising_performance}
\end{figure}
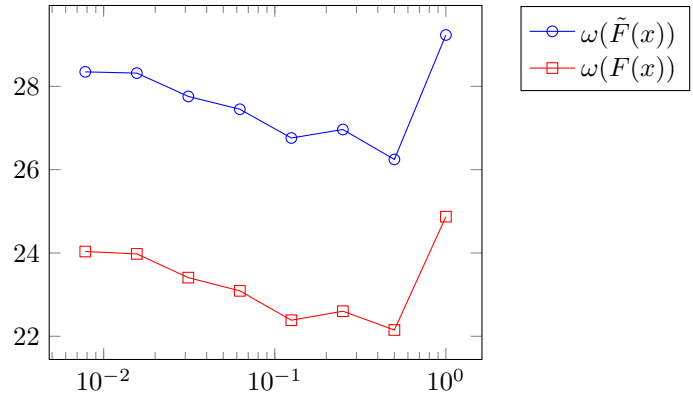
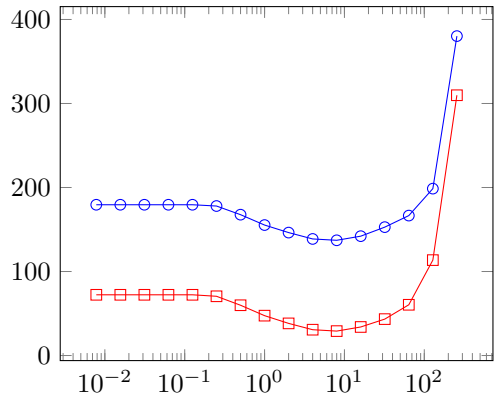
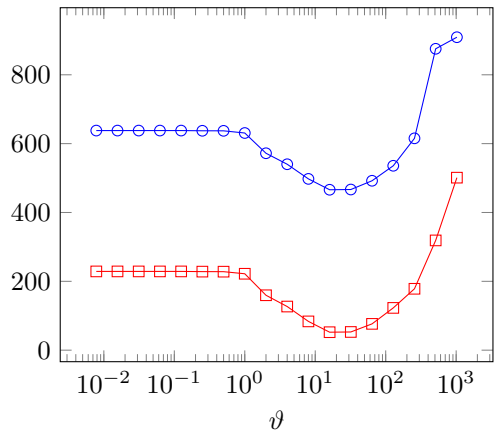

\newcommand{\imageWidth}{.3\textwidth}
\newcommand{\raiseWidth}{.15\textwidth}

\newcommand{\resultImage}[2]{
  \begin{subfigure}[t]{\imageWidth}
    \includegraphics[width=\textwidth]{Figures/Denoising/Result_#1_#2}
    \caption{$\ratioParam = \num{#2}$}
  \end{subfigure}
}

\begin{figure}[hb]
  \centering
  \begin{tabularx}{\textwidth}{Xp{\imageWidth}p{\imageWidth}X}
    &
    \begin{subfigure}[t]{\imageWidth}
      \includegraphics[width=\textwidth]{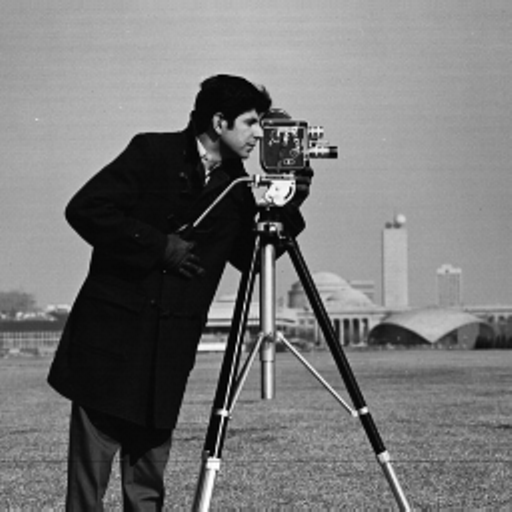}
      \caption{Original image}
      \label{figure:denoising_original}
    \end{subfigure} &
    \begin{subfigure}[t]{\imageWidth}
      \includegraphics[width=\textwidth]{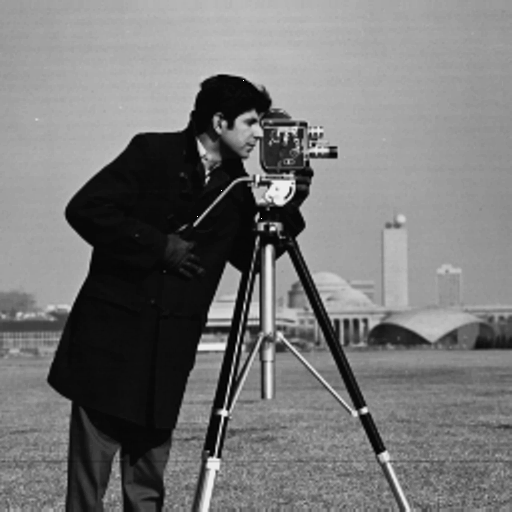}
      \caption{Noiseless reconstruction}
      \label{figure:denoising_noiseless}
    \end{subfigure}
    &
  \end{tabularx}
  \caption{Sample image for the image reconstruction.}
  \label{figure:denoising_images}
\end{figure}
\begin{figure}[hb]
  \centering
  \begin{tabularx}{\textwidth}{lp{\imageWidth}p{\imageWidth}p{\imageWidth}}
    \raisebox{\raiseWidth}{\rotatebox[origin=c]{90}{$\imageNoise = \num{0.01}$}}
    & \resultImage{0.01}{0.0} & \resultImage{0.01}{0.125} & \resultImage{0.01}{1} \\[.75cm]
    \raisebox{\raiseWidth}{\rotatebox[origin=c]{90}{$\imageNoise = \num{0.05}$}}
    & \resultImage{0.05}{0.0} & \resultImage{0.05}{8} & \resultImage{0.05}{256} \\[.75cm]
    \raisebox{\raiseWidth}{\rotatebox[origin=c]{90}{$\imageNoise = \num{0.1}$}}
    & \resultImage{0.1}{0.0} & \resultImage{0.1}{16} & \resultImage{0.1}{1024} \\
  \end{tabularx}
  \caption{Reconstructed images for different noise levels and stabilizations
    (left: no stabilization, center: best stabilization, right: large stabilization
    so that all iterates are accepted).}
  \label{figure:denoising_results}
\end{figure}

\subsection{Constrained Optimization}
\label{sec:constrained_optimization}

As a final example and in order to demonstrate the possible use of 
\cref{alg:noisy_sleqp} as a subproblem solver in constrained optimization algorithms,
we study a constrained optimization problem of the type~\eqref{eq:nonlinear_problem}.
Specifically, we examine the behavior of \cref{alg:noisy_sleqp} when applied to the
\texttt{HS71} benchmark problem of the \texttt{CUTest}~\cite{cutest} suite.
The problem is given as
\begin{equation*}
  \label{eq:HS71}
  \tag{HS71}
  \begin{aligned}
    \min_{x} \quad & x_{1}  x_{4}  (x_{1} + x_{2} + x_{3}) + x_{3} \\
    \st \quad &  x_{1}^{2} + x_{2}^{2} + x_{3}^{2} + x_{4}^{2} = 40 \\
                   &  x_{1} x_{2} x_{3} x_{4} \geq 25 \\
                   & 1 \leq x \leq 5,
  \end{aligned}
\end{equation*}
leading to suitable functions $g$ and $h$ according to~\eqref{eq:nonlinear_problem}.
The problem features of four bounded optimization variables, two
nonlinear constraints, and a nonlinear objective with an optimum at
the point $x^{*} \approx (1., 4.74, 3.82, 1.38)$ that satisfies MFCQ
and in turn the conditions for the convergence of an exacty penalty
method. As mentioned in the introduction, we solve problems of
type~\eqref{eq:nonlinear_problem}
by using the penalty function~\eqref{eq:penalty_func} with a suitable
penalization of $\nu > 0$, knowing that convergence is guaranteed for
a sufficiently large $\nu$ under mild assumptions, \ie MFCQ. Increasing
$\nu$ beyond its required value may slow down practical performance,
but convergence is maintained. Consequently, $\nu$ is often set to
a small initial value and increased when necessary~(see for example~\cite{sleqp}).

If the functions in~\eqref{eq:nonlinear_problem} are affected by
noise, the choice of $\nu$ is not as straightforward: The required
stabilization~\eqref{eq:required_stabilization} is dependent on
$\FactConst$ and therefore $\LipOmega$, which increases with $\nu$.
Similarly, the value of $\delta_{\max}$ increases with $\LipOmega$ and
therefore with $\nu$, so a large penalization has the adverse effect
of increasing the size of the critical region $C(\delta_{\max})$,
making a suitable choice of the parameter an interesting problem in
and of itself. What is more, if the constraint functions $g$ and $h$
suffer from noise, we cannot assume the iterates $x_k$ to tend towards
feasibility in the underlying noiseless problem regardless of the
value of $\nu$. 

Therefore, for our investigation, we consider a fixed value of $\nu$, which is
suitable to solve the noiseless variant of~\eqref{eq:HS71}, in our case $\nu = 100$.
We once again inject noise according to~\eqref{eq:default_noise_model}
with different values of $\epsilon_F$ and a fixed $\epsilon_{F'} = 0$. Specifically,
we run the algorithm with the choice $\epsilon_F = 10^{-2}$ and $\epsilon_F = 10^{-1}$.
Since a reasonable choice of the quadratic model would likely
require some dual estimation, we once again opt to skip quadratic
steps and instead set the iteration limit to~\num{100}.
After the algorithm has terminated, we record the criticality~$\Psi_k(1)$,
the feasibility residual,
\begin{equation*}
  \max(\|g(x_k)^{+}\|_{\infty}, \|h(x^{k})\|_{\infty}),
\end{equation*}
as well as their noisy counterparts for different values of
$\ratioParam$~(see \Cref{figure:constrained_performance}).

As was the case for the image reconstruction problem, we observe pronounced
minima of the quality metrics criticality and feasibility with respect to the
choice of $\ratioParam$. For a given choice of $\varepsilon_F$, the obtained 
minima for both quality metrics, criticality and feasibility residual,
are in close vicinity to each other.  The position of these minima is
also fairly consistent across the noisy and noiseless measurements of
both the criticality and the feasibility residuum.  Unfortunately,
setting $\ratioParam = \ratioParamOpt$ does not yield optimal results,
even though it appears as if $\ratioParamOpt$ is closer to being
optimal compared to the image reconstruction problem. Once
again, for an informed choice of $\ratioParam$, the stabilized
algorithm significantly outperforms the classical one, leading to
about an order of magnitude of reduction in terms of both criticality
and feasibility. The precise choice of the parameter and a systematic
means to determine it do, however, remain elusive.

\begin{figure}[hb]
  \centering
  \begin{tikzpicture}
  \begin{groupplot}[
    group style={
      group size=2 by 2,
      group name=plots,
      xlabels at=edge bottom,
      vertical sep=1.5cm},
    width=.6\textwidth,
    xlabel={$\ratioParam$},
    ]

    \nextgroupplot[
    legend to name={HS71CommonLegend},
    legend columns=2,
    xmode=log,
    ymode=log,
    grid style={draw=black},
    ylabel={$\epsilon_{F} = \num{1e-2}$},
    extra x ticks={7.02},
    extra x tick style={xticklabel=\empty,grid=major}
    ]

    \addplot+[mark=o, color=NoiselessColor] table[col sep=semicolon,x=Stabilization,y=Crit] {Data/HS71/Results_0.01_100.0.csv};
    \addplot+[mark=square, color=NoisyColor] table[col sep=semicolon,x=Stabilization,y=NoisyCrit] {Data/HS71/Results_0.01_100.0.csv};

    \addlegendentry{Noiseless}
    \addlegendentry{Noisy}
    
    \nextgroupplot[
    xmode=log,
    ymode=log,
    grid style={draw=black},
    extra x ticks={7.02},
    extra x tick style={xticklabel=\empty,grid=major}
    ]

    \addplot+[mark=o, draw=NoiselessColor] table[col sep=semicolon,x=Stabilization,y=FeasRes] {Data/HS71/Results_0.01_100.0.csv};
    \addplot+[mark=square, draw=NoisyColor] table[col sep=semicolon,x=Stabilization,y=NoisyFeasRes] {Data/HS71/Results_0.01_100.0.csv};

    \nextgroupplot[
    xmode=log,
    ymode=log,
    grid style={draw=black},
    ylabel={$\epsilon_{F} = \num{1e-1}$},
    extra x ticks={70.27},
    extra x tick style={xticklabel=\empty,grid=major}
    ]

    \addplot+[mark=o, draw=NoiselessColor] table[col sep=semicolon,x=Stabilization,y=Crit] {Data/HS71/Results_0.1_100.0.csv};
    \addplot+[mark=square, draw=NoisyColor] table[col sep=semicolon,x=Stabilization,y=NoisyCrit] {Data/HS71/Results_0.1_100.0.csv};

    \nextgroupplot[
    xmode=log,
    ymode=log,
    grid style={draw=black},
    extra x ticks={70.27},
    extra x tick style={xticklabel=\empty,grid=major}
    ]

    \addplot+[mark=o,draw=NoiselessColor] table[col sep=semicolon,x=Stabilization,y=FeasRes] {Data/HS71/Results_0.1_100.0.csv};
    \addplot+[mark=square,draw=NoisyColor] table[col sep=semicolon,x=Stabilization,y=NoisyFeasRes] {Data/HS71/Results_0.1_100.0.csv};

  \end{groupplot}

  \node[align=center, yshift=.5cm, inner sep=0,text width=20em,anchor=south] at (plots c1r1.north) {Criticality};
  \node[align=center, yshift=.5cm, inner sep=0,text width=20em,anchor=south] at (plots c2r1.north) {Feasibility residual};

  \path (plots c1r2.south east) -- node[midway,yshift=-1.5cm,anchor=north]{\ref{HS71CommonLegend}} (plots c2r2.south west);
\end{tikzpicture}
  \caption{
    Noiseless (red) and noisy (blue) performance of the criticality
    and feasibility residual for different values of $\ratioParam$ as well
    as the value of $\ratioParamOpt$ (vertical line).
  }
  \label{figure:constrained_performance}
\end{figure}
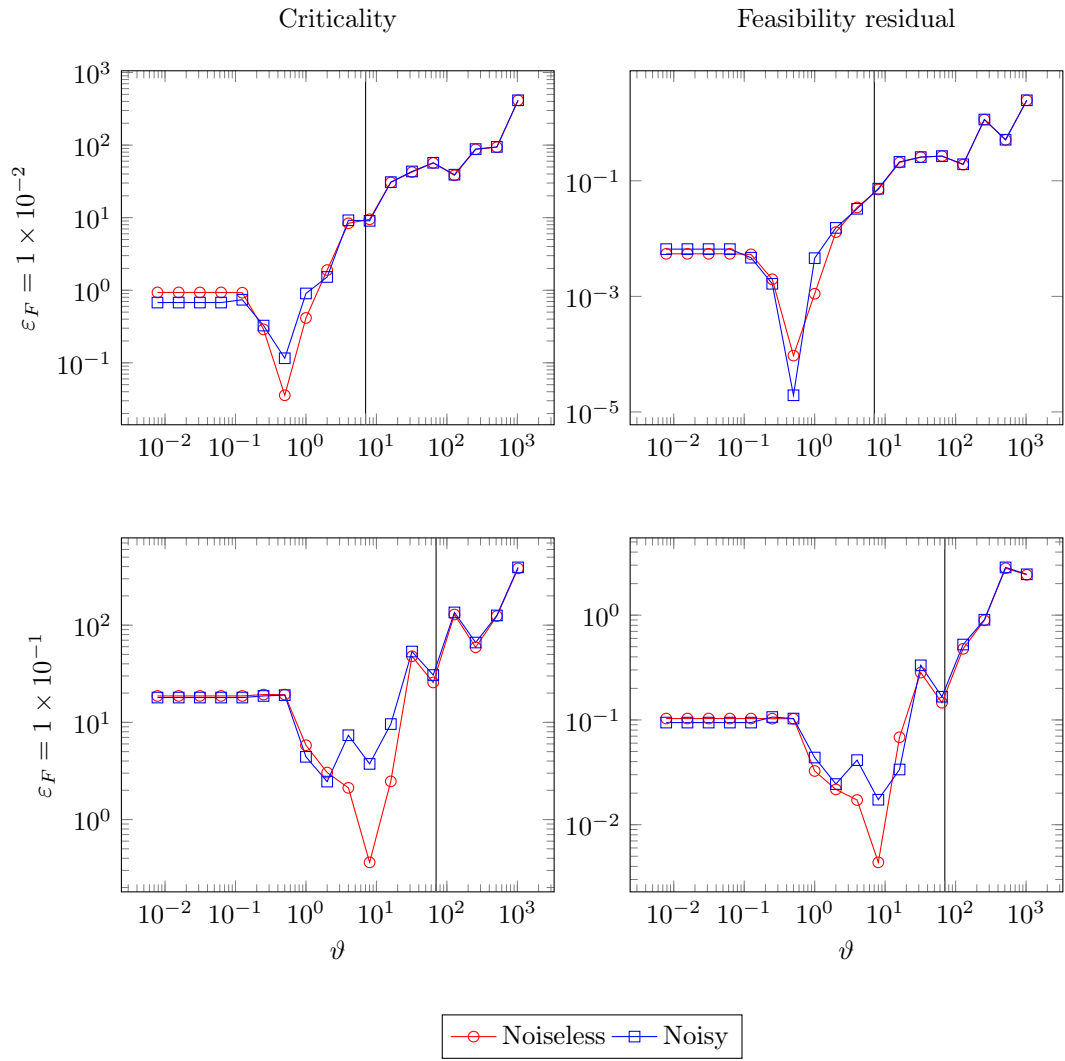

\section{Conclusion}\label{sec:conclusion}

We have presented a noise-tolerant adaptation of a well-established
trust-region method for a non-smooth optimization problem
with a structured and convex non-smoothness described by a polyhedral
function, which is therefore suitable to handling by linear programming
techniques. The adaptation only requires 
knowledge of a Lipschitz constant and bounds on the noise in the 
objective function and its derivative. The analysis of the asymptotics
of the successive linear programming algorithm can be carried out
analogously to \cite{sleqp_convergence}, where the noiseless case
is handled. As we expect from the results in \cite{noisy_trust_region},
we do not to get convergence to a first-order stationary point but a
critical region instead.

In a noiseless setting, both the behavior of the 
algorithm and its convergence properties are 
consistent and similar to previous analyses. 
The computational results show that an informed 
choice of the stabilization parameter $\ratioParam$
may improve the quality of the obtained results significantly
so that we believe it makes sense to dedicate
research to improved bounds and efficient practical 
determination strategies.

Further analysis is also needed in order to be able to use and interpret
the method as a subproblem solver for constrained optimization 
with noisy constraint and objective evaluations. In particular, it is
necessary to study the asymptotics of the feasibility residual, identify
means to control it, and classify it with respect to existing concepts
from the field of uncertainty quantification like (distributional)
chance constraints or expectation constraints.

\appendix

\section{Proofs}

In the following we give the proofs of some of the result used
in~\cref{sec:convergence}. These proofs closely follow those
in~\cite{sleqp_convergence}. We provide them here to
make this article more self-contained.
\begin{proof}[Proof of \cref{lem:critical_normalization}]
  Let $d_1 \in \R^{n}$ be a minimizer of $\noisy{\Psi}_k(1)$ for $\Delta = 1$, \ie
  $\noisy{\Psi}_k(1) = \noisy{\phi}(x_k) - \noisy{\ell}_k(d_1)$.  If
  $\Delta \geq 1$, it follows that
  \begin{equation*}
    \noisy{\Psi}_k(\Delta) \geq \noisy{\phi}(x_k) - \noisy{\ell}_k(d_1) = \noisy{\Psi}_k(1).
  \end{equation*}
  It remains to prove the case $\Delta < 1$.
  From $\normLP{d_1} \leq 1$ it follows that
  $\normLP{\Delta d_1} \leq \Delta$, \ie $\Delta d_1$ is a feasible solution with respect to
  $\Delta$. Therefore, it holds that
  \begin{equation*}
      \noisy{\Psi}_k(\Delta) = \noisy{\phi}(x_k) - \min_{\normLP{d} \leq \Delta } \noisy{\ell}_k(d)
       \geq \Delta \left[ \noisy{\phi}(x_k) - \noisy{\ell}_k(d_1) \right]
       = \Delta \noisy{\Psi}_k(1),
  \end{equation*}
  where the inequality is due to the feasibility of $\Delta d_1$ and \cref{lem:linearized_convexity}.
\end{proof}

\begin{proof}[Proof of \cref{lem:lp_step_criticality}]
  Let $d_{1}$ be a minimizer for $\Delta = 1$. Assume (towards a
  contradiction) that
  $\normLP{d_{\Delta}} < \min(\Delta, \frac{\noisy{\Psi}_k(1)}{\LipEll}) \leq \frac{\noisy{\Psi}_k(1)}{\LipEll}$. It follows from
  \cref{lem:linear_zero_convexity} that
  \begin{gather}\label{eq:noisylddelta_estimate}
    \noisy{\ell}_k(d_{\Delta}) \geq \noisy{\ell}_k(0) - \LipEll \|d_{\Delta}\| > \noisy{\ell}_k(0) - \noisy{\Psi}_k(1) = \noisy{\ell}_k(d_{1}).
  \end{gather}
  If $\Delta \geq 1$, \eqref{eq:noisylddelta_estimate} cannot hold because $d_{1}$ is feasible with
  respect to $\Delta$ and therefore cannot yield a better objective
  with respect to $\noisy{\ell}_k$ than the minimizer $d_{\Delta}$. So it must hold that
  $\normLP{d_{\Delta}} \geq \frac{\noisy{\Psi}_k(1)}{\LipEll}$ in this case $\Delta \ge 1$.
  If, on the other hand, $\Delta < 1$, then $d_{1}$ may not be
  feasible. However, since $\noisy{\ell}_k$ is convex, it holds
  for all $\lambda \in (0, 1]$ that
  \begin{equation*}
    \noisy{\ell}_k(\lambda d_{1} + (1 - \lambda) d_{\Delta})
    \leq \lambda \noisy{\ell}_k(d_{1}) + (1 - \lambda) \noisy{\ell}_k(d_{\Delta})
    \underset{\eqref{eq:noisylddelta_estimate}}< \noisy{\ell}_k(d_{\Delta}).
  \end{equation*}
  Therefore, any point on the line segment $(d_{\Delta}, d_{1}]$ has a strictly
  lower value of $\noisy{\ell}_k$ than $d_{\Delta}$. Therefore,
  no such point can be feasible with respect to the constraint on $\normLP{\cdot}$.
  Consequently, $d_{\Delta}$ must lie on the boundary of the feasible set
  implying that $\normLP{d_{\Delta}} = \Delta$.
  The result is obtained by combining these bounds.
\end{proof}

\begin{proof}[Proof of \cref{lem:model_decrease}]
  The actual step must satisfy that
  $ \noisy{q}_k(d_k) \leq \noisy{q}_k(d^C_k)$, so the first inequality
  is a given. Similarly, the last inequality is an application of
  \cref{lem:critical_normalization}.  To show that the remaining
  inequality holds, recall that the line search for the Cauchy step
  $d^{C}_{k}$ terminates with an $\alpha_k$ such that
  \begin{equation*}
    \begin{aligned}
      \noisy{\phi}(x_k) - \noisy{q}_k(d^{C}_{k}) \geq \eta \left[ \noisy{\phi}(x_k) - \noisy{\ell}_k(d^{C}_{k}) \right]
      \ge \eta\alpha_k \left[ \noisy{\phi}(x_k) - \noisy{\ell}_k(d^{\LP}_{k}) \right]
      = \noisy{\Psi}_k(\trustRadius_k^{\LP})
    \end{aligned}
  \end{equation*}
  Since $d^{C}_{k} = \alpha_k d^{\LP}_{k}$ and $d^{\LP}_{k}$ achieves $\noisy{\Psi}_k(\trustRadiusLP_k)$,
  the inequality follows from \cref{lem:linearized_convexity}.
\end{proof}

\begin{proof}[Proof of \cref{lem:cauchy_step_size_bound}]
  The first inequality is due to the fact that the Cauchy step is the
  LP step scaled by $\alpha_k$, where the LP norm of the LP step is
  bounded by $\trustRadiusLP_k$. Consider two cases for the second inequality:

  \begin{enumerate}
  \item
    The decrease condition is immediately satisfied for
    the initial step size of $\alpha_k = \min(1, \trustRadius_k / \|d^{\LP}_k\|)$.
    Consequently it follows that
    \begin{equation*}
      \begin{aligned}
        \normLP{d^{C}_k} = \normLP{\alpha_k d^{\LP}_k}
        &= \min \left(1, \trustRadius_k / \|d^{\LP}_k\| \right) \normLP{d^{\LP}_k}
      \end{aligned}
    \end{equation*}
    We consider two cases:
    \begin{enumerate}
    \item
      $ \trustRadius_k / \|d^{\LP}_k\| \geq 1$, which is to say that $\normLP{d^{C}_k} = \normLP{d^{\LP}_k}$.
      It follows from \cref{lem:lp_step_criticality} that
      \begin{equation*}
        \normLP{d^{\LP}_k} \geq \min \left( \trustRadiusLP_k, \frac{\noisy{\Psi}_k(1)}{\LipEll} \right),
      \end{equation*}
      which implies the claimed bound.
    \item
      Otherwise we know that $\normLP{d^{C}_k} = \normLP{d^{\LP}_k} \trustRadius_k / \|d^{\LP}_k\|$.
      We can use~\eqref{eq:lp_norm_equiv} to obtain that $\|d^{\LP}_k\| \leq \gamma \normLP{d^{\LP}_k}$,
      inferring that
\[
          \normLP{d^{C}_k}
          = \normLP{d^{\LP}_k} \trustRadius_k / \|d^{\LP}_k\| 
           \geq \normLP{d^{\LP}_k} \trustRadius_k / (\gamma \normLP{d^{\LP}_k} ) 
= \trustRadius_k / \gamma,
\]
      which implies the claimed bound.
    \end{enumerate}
  \item
    The decrease condition is only satisfied at a later iteration of
    the line search. Recall that the line search computes step sizes
    by multiplying a base length with powers of an input parameter
    $\tau \in (0, 1)$. We can therefore deduce that the sufficient decrease
    condition was not satisfied for $\alpha_k / \tau$ in the previous iteration, \ie
    \begin{equation*}
      \noisy{\phi}(x_k) - \noisy{q}_k(\alpha_k / \tau d^{\LP}_k) <
      \eta \left[ \noisy{\phi}(x_k) - \noisy{\ell}_k(\alpha_k / \tau d^{\LP}_k) \right]
    \end{equation*}
    Since the only difference between the linearized and quadratic model is the quadratic term, we have that
    \begin{equation*}
      \onehalf (\alpha_k / \tau)^2 \langle d^{\LP}_k, B_k d^{\LP}_k \rangle \geq
      (1 - \eta) \left[ \noisy{\phi}(x_k) - \noisy{\ell}_k(\alpha d^{\LP}_k / \tau) \right].
    \end{equation*}
    The left hand side can be bounded above by using \cref{ass:hess_bounded} and
    relation~\eqref{eq:lp_norm_equiv} to yield
    \begin{equation*}
      \begin{aligned}
        \:\onehalf (\alpha_k / \tau)^2 \langle d^{\LP}_k, B_k d^{\LP}_k \rangle
        \leq& \: \onehalf (\alpha_k / \tau)^2 \beta \gamma^2 \normLP{d^{\LP}_k}^2 \\
        \leq& \: \onehalf (\alpha_k / \tau)^2 \beta \gamma^2 \normLP{d^{\LP}_k} \trustRadiusLP_k. \\
      \end{aligned}
    \end{equation*}
    Similarly, for the right hand side we can use
    \cref{lem:linearized_convexity,lem:critical_normalization}
    to obtain
    \begin{equation*}
      \begin{aligned}
        \left[ \noisy{\phi}(x_k) - \noisy{\ell}_k(\alpha / \tau d^{\LP}_k ) \right]
        \geq & \alpha_k / \tau \left[ \noisy{\phi}(x_k) - \noisy{\ell}_k(d^{\LP}_k) \right] \\
        \geq & \alpha_k / \tau \min(1, \trustRadiusLP_k) \noisy{\Psi}_k(1),
      \end{aligned}
    \end{equation*}
    Putting these inequalities together yields the bound
    \begin{equation*}
      \normLP{d^{C}_k} = \alpha_k \normLP{d^{\LP}_k} \geq \frac{2 (1 - \eta) \tau}{ \beta \gamma^{2}}
      \min \left( 1, \frac{1}{\trustRadiusLP_k} \right) \noisy{\Psi}_k(1)
    \end{equation*}
    required to complete the proof.\qedhere
  \end{enumerate}
\end{proof}
\begin{proof}[Proof of \cref{lem:finite_successful_convergence}]
  Let $k_0$ be the index of the last accepted step. Then,
  $x_{k + 1} = x_{k} \enifed x^{*}$ for all $k > k_0$. Consequently,
  after finishing the $k_0$-the iteration, $\noisy{\Psi}_k(1)$ stays
  at a constant value of $\delta \geq 0$.  What is more, following
  iteration $k_0$ we have that
  $\trustRadius_{k + 1} < \kappa_{u} \Delta_{k}$, where
  $\kappa_{u} < 1$. Therefore, $\trustRadius_k$ tends to zero.  Recall
  from \cref{lem:tr_down} that if $\delta > 0$, then $\trustRadius_k$
  is bounded away from zero. Therefore, since $\trustRadius_k$ tends
  to zero, it must hold that $\delta = 0$.
\end{proof}

\section{Estimations}
\label{sec:estimations}

\paragraph{Lipschitz constant of the $\ell^1$-penalty function}

In the following, we give an estimation for the Lipschitz constant
$\LipOmega$ of the penalty function
$\omega : \R \times \R^{m} \to \R$,
$\omega(x, y) = x + \nu \|y\|_{1}$, based on the constant $\nu > 0$
and the dimension $m \in \N$. Since we use this function in
all of the examples in~\Cref{sec:numerics}, and since the value of
$\ratioParamOpt$ depends on the value of $\LipOmega$, we make its
derivation explicit. To obtain an optimal
value of $\LipOmega$, we solve the optimization problem
\begin{equation*}
  \begin{aligned}
    \max_{x, y, x', y'} \quad & |\omega(x', y') - \omega(x, y)| \\
    \st \quad & (x - x')^{2} + \|y - y'\|^{2} \leq 1,
  \end{aligned}
\end{equation*}
\ie we maximize the difference in values of $\omega$ while controlling
the distance between the points $(x, y)$ and $(x', y')$. Observe
that
\begin{equation*}
  |\omega(x', y') - \omega(x, y)| = |(x - x') + \nu \|y - y'\|_{1}|,
\end{equation*}
from which it follows that both the objective and the constraint value
only depend on $x - x'$ and $y - y'$. We can therefore simplify the problem
by setting $y' = 0$ and $y'= 0$:
\begin{equation*}
  \begin{aligned}
    \max_{x, y} \quad & |x + \nu \|y\|_{1}| \\
    \st \quad & x^{2} + \|y\|^{2} \leq 1.
  \end{aligned}
\end{equation*}
We can simplify the problem further by realizing that we can assume
both $x$ and $y$ to be non-negative, eliminating the absolute value in
the objective. The largest ratio of $\nu \|y\|_{1}$ over $\|y\|^{2}$
is achieved by setting all entries of $y$ to the same value
$y_0 \in \R$, yielding the problem
\begin{equation*}
  \begin{aligned}
    \max_{x, y_0} \quad & x + \nu m y_0 \\
    \st \quad & x^{2} + m y_0^{2} \leq 1 \\
    & x, y_0 \geq 0.
  \end{aligned}
\end{equation*}
By setting $z \define \sqrt{m} y_0$, we obtain the problem
\begin{equation*}
  \begin{aligned}
    \max_{x, y_0} \quad & x + \nu \sqrt{m} z \\
    \st \quad & x^{2} + z^{2} \leq 1 \\
                        & x, z \geq 0.
  \end{aligned}
\end{equation*}
The optimal solution of this problem is attained at
\begin{equation*}
  \begin{pmatrix}
    x^{\ast} \\
    z^{\ast}
  \end{pmatrix} = \frac{1}{\sqrt{1 + \nu^{2}m}}
  \begin{pmatrix}
    1 \\
    \nu \sqrt{m}
  \end{pmatrix},
\end{equation*}
yielding the objective $\sqrt{1 + \nu^{2}m} \enifed \LipOmega$.

\paragraph{Image Reconstruction}
In the following, we provide estimations regarding the noise levels
associated with the image fidelity map introduced in
\Cref{sec:image_denoising}. Recall that the squared Frobenius norm
of a matrix $A \in \R^{M \times N}$ is given by
$\normFrob{A}^2 \define \sum_{i = 1}^{M} \sum_{j = 1}^{N} a_{ij}^{2}$.
Thus, if $|a_{ij}| \leq \epsilon$ for a given $\epsilon > 0$,
it follows that
\begin{equation*}
  \normFrob{A}^2 \leq \sum_{i = 1}^{M} \sum_{j = 1}^{N} \epsilon^{2} = \epsilon^{2} M N,
\end{equation*}
and therefore, that $\normFrob{A} \leq \epsilon \sqrt{MN}$. The noisy
fidelity function $\noisy{F}(X)$ satisfies the identities
\begin{equation*}
  \begin{aligned}
    \noisy{F}(X) &= \frac{1}{2} \normFrob{X - \noisy{Y}}^{2} = \frac{1}{2} \normFrob{X - (Y + \delta_{Y}(X))}^{2} \\
&= \frac{1}{2} \normFrob{X - Y}^{2} + \langle X - Y, \delta_{Y}(X)\rangle_F + \frac{1}{2} \normFrob{\delta_{Y}(X)}^{2},
  \end{aligned}
\end{equation*}
where $\langle\cdot,\cdot\rangle_F$
denotes the inner product that
induces the Frobenius norm.
Since the entries of $X$ and $Y$ are in $[0, 1]$ , and therefore have
absolute values bounded by \num{1}, it follows that
$\normFrob{X - Y} \leq \sqrt{M N}$.
The choice of distribution implies that
the values in $\delta_{Y}(X)$ are bounded by $ \pm\imageNoise$, and
therefore that $\normFrob{\delta_{Y}(X)} \leq \imageNoise \sqrt{M N}$, from which it follows that
\begin{equation*}
  |\noisy{F}(X) - F(X)|
  \leq \normFrob{X - Y} \normFrob{\delta_{Y}(X)} + \frac{1}{2} \normFrob{\delta_{Y}(X)}^{2}
  \leq \left(\imageNoise + \frac{1}{2} \imageNoise^{2} \right) MN,
\end{equation*}
by means of the Cauchy--Schwarz
inequality for $\langle\cdot,\cdot\rangle_F$. That is the noise level $\imageNoise$ yields a
corresponding value for $\epsilon_F$. Similarly, it holds that
$\noisy{F}'(X) = F'(X) - \noisy{Y}$, and therefore
$\normFrob{\noisy{F}'(X) - F'(X)} \leq \normFrob{\delta_{Y}(X)} \leq \imageNoise \sqrt{M N}$,
corresponding to a value for $\epsilon_F'$.

\emergencystretch 2em
\printbibliography

\end{document}